\documentclass[onefignum,onetabnum]{siamart220329}
\usepackage{graphics,graphicx,epstopdf,url}
\usepackage{arydshln}
\usepackage{color}
\usepackage{geometry}
\usepackage{algorithmic}
\usepackage{url,cases,supertabular}
\usepackage{amssymb,mathtools}
\usepackage{enumerate,makecell}
\usepackage{amsfonts}
\usepackage{graphicx,epstopdf}
\usepackage{color,verbatim}
\usepackage{mathrsfs,subeqnarray,subfigure}
\usepackage{listings,array}
\usepackage{booktabs}
\usepackage{dashrule}
\usepackage{arydshln}
\usepackage{graphicx}
\usepackage{amsfonts}
\usepackage{mathrsfs}
\usepackage{graphicx}  \usepackage{epstopdf}
\usepackage{subfigure}  \usepackage{lineno}
\usepackage{multirow,bm}
\usepackage{color}
\usepackage{algorithm}
\usepackage{algorithmic}
\usepackage{longtable}
\usepackage{mdwtab}
\usepackage{mdwmath}
\usepackage{bbding}
\usepackage{booktabs}
\usepackage{rotating}
\usepackage{enumitem}
\usepackage{lineno}
\usepackage{url}
\usepackage{pdflscape}
\usepackage{makecell}

\newcommand{\Mcal}{\mathcal{M}}
\usepackage{pifont}

\numberwithin{equation}{section}

\usepackage{essay-def-siam}
\usepackage{tablefootnote}

\usepackage{lipsum}
\usepackage{clrscode}
\usepackage{appendix}
\usepackage{bm}
\usepackage{booktabs}
\usepackage{url}
\usepackage{multirow}
\usepackage{textcomp}
\usepackage{amsmath}
\usepackage{amsfonts}
\usepackage{amssymb}
\usepackage{mathrsfs}
\usepackage{hyperref}
\usepackage{graphicx,graphics,subfigure}
\usepackage{hyperref,url}
\usepackage{epsf,epstopdf}
\usepackage{algorithm}
\usepackage{algorithmic}
\usepackage{bbm}
\usepackage{dsfont}
\usepackage{indentfirst}
\usepackage{pdfpages}
\usepackage{pdflscape}
\usepackage{longtable}
\ifpdf
  \DeclareGraphicsExtensions{.eps,.pdf,.png,.jpg}
\else
  \DeclareGraphicsExtensions{.eps}
\fi
\allowdisplaybreaks[1]

\headers{A single loop $\mathcal{O}(\epsilon^{-3})$ stochastic smoothing algorithm}{Deng and Peng and Wu}

\title{Single-loop $\mathcal{O}(\epsilon^{-3})$ stochastic smoothing algorithms for nonsmooth optimization problem on Riemannian manifold}

\author{
Kangkang Deng\thanks{ Department of Mathematics,  National University of Defense Technology, Changsha, 410073,
China (\email{freedeng1208@gmail.com}).}
\and
Zheng Peng\thanks{Corresponding author. School of Mathematics and Computational Science, Xiangtan University, Xiangtan, 411105, china
(\email{pzheng@xtu.com}).} 
\and Weihe Wu\thanks{School of Microelectronics, Fudan University, Shanghai, 200433, china
(\email{1009248350@qq.com}).}  
}

\usepackage{amsopn}

\ifpdf
\hypersetup{
  pdftitle={Decentralized projected Riemannian gradient method},
  pdfauthor={Kangkang Deng}
}
\fi

\allowdisplaybreaks[2]
\begin{document}

\maketitle
\vspace{0.5em}
\begin{center}
    \textit{May 15, 2025 (revised: May 26, 2025)}
\end{center}
\vspace{0.5em}
\begin{abstract}
In this paper, we develop two Riemannian stochastic smoothing algorithms for nonsmooth optimization problems on Riemannian manifolds, addressing distinct forms of the nonsmooth term \( h \). Both methods combine dynamic smoothing with a momentum-based variance reduction scheme in a fully online manner. When \( h \) is Lipschitz continuous, we propose an stochastic algorithm under adaptive parameter that achieves the optimal iteration complexity of \( \mathcal{O}(\epsilon^{-3}) \), improving upon the best-known rates for exist algorithms. When \( h \) is the indicator function of a convex set, we design a new algorithm using truncated momentum, and under a mild error bound condition with parameter \( \theta \geq 1 \), we establish a complexity of \( \tilde{\mathcal{O}}(\epsilon^{-\max\{\theta+2, 2\theta\}}) \), in line with the best-known results in the Euclidean setting. Both algorithms feature a single-loop design with low per-iteration cost and require only \( \mathcal{O}(1) \) samples per iteration, ensuring that sample and iteration complexities coincide. Our framework encompasses a broad class of problems and recovers or matches optimal complexity guarantees in several important settings, including smooth stochastic Riemannian optimization, composite problems in Euclidean space, and constrained optimization via indicator functions.
\end{abstract}


% REQUIRED

\section{Introduction}

In this paper, we consider a general nonsmooth optimization problem of expectation (online) minimization over Riemannian manifold $\mcM$, defined as
\begin{equation}\label{general problem}
\min_{x\in\mcM} F(x):= \mbE_{\xi\in\mathcal{D}}[\tilde{f}(x,\xi)]+h(c(x)), 
\end{equation}
where $\mcM$ is a compact Riemannian manifold embedded in $\mbR^n$, $h:\mbR^m \rightarrow\mbR$ is a nonsmooth weakly-convex function with a readily available proximal operator, and $c:\mbR^n \rightarrow \mbR^m$ is a nonlinear operator, $\xi$ is a random variable with sample space $\mcD$.   We define $f(x): = \mbE[\tilde{f}(x,\xi)]$, where each realization $\tilde{f}(x,\xi)$ is assumed to be continuously differentiable.   The model \eqref{general problem} encompasses a wide range of applications, including sparse principal component analysis (SPCA) \cite{jolliffe2003modified}, range-based independent component analysis \cite{selvan2013spherical, selvan2015range} and robust low-rank matrix completion \cite{cambier2016robust, hosseini2017riemannian}, see 
 \cite{absil2017collection,hu2020brief} for more examples.

In the case where $h = 0$, problem \eqref{general problem} reduces to a smooth stochastic optimization problem over a manifold. This special case has been extensively studied. One of the earliest methods for solving \eqref{general problem} is Riemannian stochastic gradient descent (RSGD) algorithm \cite{bonnabel2013stochastic}, which achieves an iteration complexity of $\mathcal{O}(\epsilon^{-4})$ for finding an $\epsilon$-stationary point in term of
  \begin{equation}
      \mathbb{E}[ \| \grad f(x)\|  ] \leq \epsilon.
  \end{equation}
  To reduce the complexity bounds, several works \cite{kasai2018riemannian,sato2019riemannian,zhou2019faster,han2021improved} have extended RSGD by incorporating variance reduction techniques. Recent works \cite{ijcai2021p345,demidovich2024streamlining} aim to avoid the need for double-loop structures and large batch sizes by employing stochastic recursive momentum and probabilistic gradient estimation techniques.

When \( \mathcal{M} = \mathbb{R}^n \) or is a convex set amenable to efficient projection, the problem reduces to a nonsmooth optimization problem in Euclidean space, for which a wide array of algorithms have been proposed. In cases where \( c(x) = x \) (the identity operator) or a linear operator, numerous proximal gradient methods and ADMM-type algorithms—such as those in \cite{ghadimi2016accelerated,tran2022hybrid,gao2024non,xu2023momentum,wang2019spiderboost} and \cite{huang2016stochastic,huang2019faster,deng2025stochastic}—have been developed. For the more general setting where \( c(x) \) is an arbitrary operator and \( h \) corresponds to the indicator function of a convex set \( C \), the problem becomes a stochastic constrained optimization problem. Algorithmic approaches to such problems broadly fall into two categories. The first involves penalty methods \cite{wang2017penalty,alacaoglu2024complexity,lu2024variance}, and augmented Lagrangian method (ALM)-based frameworks \cite{shi2025momentum,li2024stochastic}, which reformulate the problem into a sequence of unconstrained subproblems solvable via stochastic gradient techniques. The second category encompasses sequential quadratic programming (SQP) algorithms \cite{berahas2021sequential,curtis2024worst,curtis2021inexact}, which iteratively approximate the problem using quadratic programming subproblems.

Recently, several methods have been proposed for solving problem with $h\neq 0$ and $\mcM$ is a general compact submanifold.  In particular, Li et al. \cite{li2021weakly} focus on a class of weakly convex problems on the Stiefel manifold. They present a Riemannian stochastic subgradient method and show that it has an iteration complexity of $\mcO(\epsilon^{-4})$ for driving a natural stationarity measure below $\epsilon$. Peng et al. \cite{peng2022riemannian} propose a Riemannian stochastic smoothing method and give an iteration complexity of $\mathcal{O}(\epsilon^{-5})$ for driving a natural stationary point.    Wang et al. \cite{wang2022riemannian} proposes an Riemannian stochastic proximal gradient method for minimizing a nonsmooth function over the Stiefel manifold. They show that the proposed algorithm finds an $\epsilon$-stationary point with $\mathcal{O}(\epsilon^{-3})$ first-order oracle complexity. It should be emphasized that each oracle call involves a subroutine whose iteration complexity remains unclear.  Recent work by \cite{deng2024oracle} introduced a novel augmented Lagrangian framework enhanced with a recursive momentum-based variance reduction technique for problem \eqref{general problem} when $c(x) = \mathcal{A}x$ is a linear operator. This approach achieves an improved iteration complexity of $\mathcal{O}(\epsilon^{-3.5})$ to attain an $\epsilon$-approximate KKT point of \eqref{general problem}.

\subsection{Contributions}

\begin{itemize}
    \item

We propose two Riemannian stochastic smoothing algorithms tailored to different forms of the nonsmooth term $h$. Both algorithms integrate dynamic smoothing techniques with an online momentum-based variance reduction scheme. When $h$ is Lipschitz continuous, we develop a stochastic first-order algorithm with adaptive parameters (refer to Algorithm~\ref{alg:stomanial}), which achieves the optimal iteration complexity of $\mathcal{O}(\epsilon^{-3})$. This result improves upon the best-known complexities of existing stochastic algorithms in the Riemannian setting \cite{peng2022riemannian, li2021weakly, deng2024oracle}, and, notably, matches the best-known complexity bounds for deterministic Riemannian algorithms \cite{beck2023, deng2024oracle}. In the case where $h$ is the indicator function of a convex set, we propose a stochastic first-order algorithm (refer to Algorithm \ref{alg:stomanial-1}) based on a truncated momentum estimator. Under a mild error bound condition with parameter $\theta \geq 1$, we establish an iteration complexity of $\tilde{\mathcal{O}}(\epsilon^{-\max\{\theta+2, 2\theta\}})$, which aligns with the best-known complexity results in the Euclidean setting \cite{lu2024variance}.

    \item The proposed algorithms enjoys several favorable properties. Firstly,  both algorithms adopt a single-loop design without requiring any nested subroutines or inner loops. As a result, each iteration has low per-iteration computational cost. Secondly,  Each iteration requires only  $\mathcal{O}(1)$ samples, without relying on large mini-batches or periodic restarts. This online nature guarantees that the sample complexity and iteration complexity are  the same. Finally, our framework accommodates a wide range of problems and achieves optimal (or near-optimal) complexity in several important special cases: When $h \equiv 0$, our method recovers the smooth stochastic Riemannian optimization setting in~\cite{han2020riemannian}. 
  For the case that $\Mcal = \mathbb{R}^n$  and $h$ is a convex regularizer, our results match those of stochastic proximal gradient or stochastic ADMM algorithms in~\cite{xu2023momentum,deng2025stochastic}.
 If we consider $h$ is an indicator function, our complexity aligns with those obtained in~\cite{lu2024variance,shi2025momentum}.
    
\end{itemize}

In Table \ref{tab1}, we summarize our complexity results and several existing methods to produce
an $\epsilon$-stationary point of problem \eqref{general problem}.  The notation $\tilde{\mathcal{O}}$ denotes the dependence of a $\log$ factor.       It can easily be shown that our algorithms achieve better oracle complexity results.

\begin{table} 
   \caption{Comparison of the oracle complexity results of several methods in the literature to our method to produce an $\epsilon$-stationary point. ``compact'' means that $\Mcal$ is compact submanifold, ``Lipschitz'' denotes that $h$ is Lipschitz continuous, ``indicator'' means that $h$ is an indicator function over convex set, ``Complexity'' means iteration complexity. }
    \centering
    \small
    \begin{tabular}{|c|c|c|c|c|c|}
  \hline
  Algorithms & $c(x)$ & Manifolds & Nonsmooth & Single-loop  & Complexity \\ \hline
    subgradient \cite{li2021weakly}  &  --- &  Stiefel    &  Lipschitz  & \checkmark    &  $ \mathcal{O}(\epsilon^{-4})$ \\ \hline
    R-ProxSPB \cite{wang2022riemannian} & identity & Stiefel  & Lipschitz  & \ding{55}   & $\mathcal{O}(\epsilon^{-3})$\tablefootnote{This algorithm needs a subroutine, the overall operation complexity is unclear.} \\ \hline
 smoothing \cite{peng2022riemannian} &  linear & compact   & Lipschitz & \ding{55}    &  $\mathcal{O}(\epsilon^{-5})$ \\ \hline
stoManIAL \cite{deng2024oracle} & linear  & compact  & Lipschitz & \ding{55} &     $\tilde{\mathcal{O}}(\epsilon^{-3.5})$ \\
  \hline
   \multirow{2}{*}{ \textbf{this paper}} & \multirow{2}{*}{nonlinear} &  \multirow{2}{*}{compact } & Lipschitz  &  \multirow{2}{*}{\checkmark }  &   $\mathcal{O}(\epsilon^{-3})$\\
 & & &  indicator & &     $\tilde{\mathcal{O}}(\epsilon^{-\max\{\theta+2, 2\theta\}})$\tablefootnote{The parameter $\theta \geq 1$ comes from the error bound condition in Assumption \ref{assum:error-bound}.}\\
  \hline
\end{tabular}
 \label{tab1}
\end{table}

\subsection{Notations} 

Throughout, Euclidean space, denoted by $\mbR^n$, equipped with an inner product $\left<\cdot, \cdot\right>$ and inducing norm $\|x\| = \sqrt{\left<x, x\right>}$. Given a matrix $A$, we use $\|A\|_F$ to denote the Frobenius norm, $\|A\|_1:=\sum_{ij}\lvert A_{ij}\vert$ to denote the $\ell_1$ norm. For a vector $x$, we use $\|x\|_2$ and $\|x\|_1$ to denote Euclidean norm and $\ell_1$ norm, respectively. The indicator function of a set $\mcC$, denoted by $\delta_{\mcC}$, is set to be zero on $\mcC$ and $+\infty$ otherwise. The distance from $x$ to $\mcC$ is denoted by $\mathrm{dist}(x, \mcC): = \min_{y\in\mcC}\|x-y\|$. For a differentiable function $f$ on $\mcM$, let $\grad f(x)$ be its Riemannian gradient. If $f$ can be extended to the ambient Euclidean space, we denote its Euclidean gradient by $\nabla l(x)$.  We denote $l(t)=\mcO(l(t))$ if there exists a positive constant $M$ and $t_0$ such that $l(t)\leq Ml(t)$ for all $t\geq t_0$. We use $\tilde{\mcO}(\cdot)$ to hide poly-logarithmic factors.

\section{Preliminary} 

\subsection{Riemannian optimization} 
 An $n$-dimensional smooth manifold $\mathcal{M}$ is an   $n$-dimensional topological manifold equipped with a smooth structure, where each point has a neighborhood that is diffeomorphism to the $n$-dimensional Euclidean space. For all $x\in\mathcal{M}$, there exists a chart $(U,\psi)$ such that $U$ is an open set and $\psi$ is a diffeomorphism between $U$ and an open set $\psi(U)$ in the Euclidean space.   A tangent vector $\eta_x$ to $\mathcal{M}$ at $x$ is defined as tangents of parametrized curves $\alpha$ on $\mathcal{M}$ such that $\alpha(0) = x$ and
\[\nonumber
  \eta_x u : = \dot{\alpha}(0)u = \left.\frac{d(u(\alpha(t)))}{dt} \right\vert_{t=0} , \forall u\in  \wp_x\mathcal{M},
\]
where $\wp_x\mathcal{M}$ is the set of all real-valued functions $f$ defined in a neighborhood of $x$ in $ \mathcal{M}$. Then, the tangent space $T_x\mathcal{M}$ of a manifold $\mathcal{M}$ at $x$ is defined as the set of all tangent vectors at point $x$.  The manifold $\mathcal{M}$ is called a Riemannian manifold if it is equipped with an inner product on the tangent space $T_x\mathcal{M}$ at each $x\in\mathcal{M}$. If $\mathcal{M}$ is a Riemannian submanifold of an Euclidean space $\mathcal{E}$, the inner product is defined as the Euclidean inner product: $\left<\eta_x,\xi_x\right> = \mathrm{tr}(\eta_x^\top \xi_x)$. The Riemannian gradient $\grad  f(x) \in T_x\mathcal{M}$ is the unique tangent vector satisfying
$$  \left< \grad f(x), \xi \right> = {D}f(x)[\xi], \forall \xi\in T_x\mathcal{M}, $$
{where $Df$ denotes the differential of $f$ in $\mathbb{R}^n$.}   If $\mathcal{M}$ is a compact Riemannian manifold embedded in an Euclidean space, we have that $\grad f(x) = \mathcal{P}_{T_x\mathcal{M}}(\nabla f(x))$, where $\nabla f(x)$ is the Euclidean gradient, $\mathcal{P}_{T_x \mathcal{M}}$ is the projection operator onto the tangent space $T_x \mathcal{M}$. 
 The retraction operator is one of the most important ingredients for manifold optimization, which turns an element of $T_x\mathcal{M}$ into a point in $\mathcal{M}$.

\begin{definition}[Retraction, \cite{AbsMahSep2008}]\label{def-retr}
  A retraction on a manifold $\mathcal{M}$ is a smooth mapping $\mathcal{R}:T\mathcal{M}\rightarrow \mathcal{M}$ with the following properties. Let $\mathcal{R}_x:T_x\mathcal{M} \rightarrow \mathcal{M}$ be the restriction of $\mathcal{R}$ at $x$. It satisfies
\begin{itemize}
  \item $\mathcal{R}_x(0_x) = x$, where $0_x$ is the zero element of $T_x\mathcal{M}$,
  \item ${D}\mathcal{R}_x(0_x) = id_{T_x\mathcal{M}}$,where $id_{T_x\mathcal{M}}$ is the identity mapping on $T_x\mathcal{M}$.
\end{itemize}
\end{definition}
We have the following Lipschitz-type inequalities on the retraction on the compact submanifold.
\begin{proposition}[{\cite[Appendix B]{grocf}}]
    Let $\mathcal{R}$ be a retraction operator on a compact submanifold $\Mcal$. Then, there exist two positive constants $\alpha, \beta$ such that
    for all $x\in \mathcal{M}$ and  all $u \in T_{x}\mathcal{M}$, we have
   \begin{equation}\label{eq:retrac-lipscitz}
  \left\{ \begin{aligned}
     \|\mathcal{R}_x(u) - x\| &\leq \alpha \|u\|, \\
     \|\mathcal{R}_x(u) - x - u\| &\leq \beta \|u\|^2.
   \end{aligned}\right.
   \end{equation}
\end{proposition}

Another basic ingredient for manifold optimization is vector transport.
\begin{definition}[Vector Transport, \cite{AbsMahSep2008}]
 The vector transport $\mathcal{T}$ is a smooth mapping
 \begin{equation}
   T\mathcal{M}\oplus T\mathcal{M} \rightarrow T\mathcal{M}:(\eta_x,\xi_x)\mapsto \mathcal{T}_{\eta_x}(\xi_x)\in T\mathcal{M}
 \end{equation}
 satisfying the following properties for all $x\in\mathcal{M}$:
 \begin{itemize}
   \item for any $\xi_x\in T_x\mathcal{M}$, $\mathcal{T}_{0_x}\xi_x = \xi_x$,
   \item $\mathcal{T}_{\eta_x}(a\xi_x + b\tau_x)  = a \mathcal{T}_{\eta_x}(\xi_x) + b\mathcal{T}_{\eta_x}(\tau_x)$.
 \end{itemize}
 When there exists $\mathcal{R}$ such that $y=\mathcal{R}_{x}(\eta_x)$, we write $\mathcal{T}_{x}^y(\xi_x) = \mathcal{T}_{\eta_x}(\xi_x)$. 
\end{definition}

\subsection{Moreau envelope and retraction smoothness}
We first provide the definition of proximal operator and Moreau envelope. 
\begin{definition}\label{def:moreau}
 For a proper,  convex and lower semicontinuous function $h:\mbR^m \rightarrow \mbR$, the Moreau envelope of $h$ with the parameter $\mu>0$ is given by
 \begin{equation*}
 h_{\mu}(y): = \inf_{z\in\mbR^m} \left\{h(z) + \frac{1}{2\mu }\|z- y\|^2 \right\}. 
 \end{equation*}
 The optimal solution is given by the proximal operator:
 \begin{equation}\label{def:prox}
   \prox_{\mu h}(y) = \arg\min_{z\in\mbR^m} \left\{h(z) + \frac{1}{2\mu }\|z- y\|^2 \right\}.
 \end{equation}
 %Moreover, the gradient of $h_{\mu}$ is given by $ \nabla h(y) = \frac{1}{\mu}(y - \prox_{\mu h}(y))$.
\end{definition}

The following result demonstrates that the Moreau envelope of a  convex function is not only continuously differentiable but also possesses a bounded gradient norm. 

\begin{proposition}\label{h proposiztion}\cite[Lemma 3.3]{bohm2021variable}
Let $h$ be a proper,  convex function. Denote $\mu >0$. Then, Moreau envelope $h_{\mu}$ has Lipschitz continuous gradient over $\mbR^n$ with constant $\mu^{-1}$, and the gradient is given by 
 \begin{equation}\label{h gradient}
 \nabla h_{\mu}(x) = \frac{1}{\mu}\left(x - \prox_{\mu h}(x) \right) \in \partial h\left(\prox_{\mu h}(x)\right).
 \end{equation}
 Moreover, if $h$ is $\ell_h$-Lipschitz continuous, it holds that
\begin{equation} \label{h bound}
\|\nabla h_{\mu}(x)\| \leq \ell_h, \; \|x - \prox_{\mu h}(x) \| \leq \mu \ell_h. 
\end{equation}
\end{proposition}

The following lemma, stated as Lemma 4.1 in \cite{bohm2021variable}, establishes a relationship between the function values of two Moreau envelopes with distinct parameters. 
\begin{lemma}\label{h Lipschitz}\cite[Lemma 4.1]{bohm2021variable} 
 Let $h$ be a proper, closed, and  convex function. Then 
 \begin{equation}\label{eq:moreau-bound}
 h_{\mu_2}(y) \leq h_{\mu_1}(y) + \frac{1}{2} \frac{\mu_1 - \mu_2}{\mu_2} \mu_1 \left\|\nabla h_{\mu_1}(y)\right\|^2, 
 \end{equation}
 where $\mu_1$ and $\mu_2$ satisfy that $0< \mu_2 \leq \mu_1 $. In particular, consider the following two cases:
 \begin{itemize}
     \item When $h$ is Lipschitz continuous with module $\ell_h$, it holds that
     \begin{equation}\label{eq:moreau-bound-1}
 h_{\mu_2}(y) \leq h_{\mu_1}(y) + \frac{1}{2} \frac{\mu_1 - \mu_2}{\mu_2} \mu_1 \ell_h^2, 
 \end{equation}
 \item When $h$ is a indicator function on convex set $\mathcal{C}$, i.e., $h(x) = \delta_{\mathcal{C}}(x)$, it holds that
 \begin{equation}\label{eq:moreau-bound-2}
 h_{\mu_2}(y) \leq h_{\mu_1}(y) + \frac{1}{2} ( \frac{1}{\mu_2} - \frac{1}{\mu_1} )   \mathrm{dist}^2(y,\mathcal{C}).
 \end{equation}
 \end{itemize}
\end{lemma}

\begin{proof}
    By using the definition of the Moreau envelope, we obtain
$$
\begin{aligned}
 h_{\mu_2}(y) & =\min _{u \in \mathbb{R}^n}\left\{h(u)+\frac{1}{2 \mu_2}\|y-u\|^2\right\} \\
& =\min _{u \in \mathbb{R}^n}\left\{h(u)+\frac{1}{2 \mu_1}\|y-u\|^2+\frac{1}{2}\left(\frac{1}{\mu_2}-\frac{1}{\mu_1}\right)\|y-u\|^2\right\} \\
& \leq h\left(\operatorname{prox}_{\mu_1 h}(y)\right)+\frac{1}{2 \mu_1}\left\|y-\operatorname{prox}_{\mu_1 h}(y)\right\|^2+\frac{1}{2}\left(\frac{1}{\mu_2}-\frac{1}{\mu_1}\right)\left\|y-\operatorname{prox}_{\mu_1 h}(y)\right\|^2 \\
& = h_{\mu_1}(y)+\frac{1}{2}\left(\frac{\mu_1-\mu_2}{\mu_2}\right) \mu_1\left\|\nabla h_{\mu_1}(y)\right\|^2,
\end{aligned}
$$
which proves \eqref{eq:moreau-bound}. When $h$ is Lipschitz continuous with $\ell_h$, we obtain \eqref{eq:moreau-bound-1} by applying \eqref{h bound}. If $h(x) = \delta_{\mathcal{C}}(x)$, since $\left\|y-\operatorname{prox}_{\mu_1 h}(y)\right\|^2 = \left\|y-\mathcal{P}_{\mathcal{C}}(y)\right\|^2$, which implies \eqref{eq:moreau-bound-2}.
\end{proof}

Next, we provide the definition of retraction smoothness. This concept plays a crucial role in the convergence analysis of the algorithms proposed in the subsequent sections.
\begin{definition}\cite[Retraction smooth]{grocf}\label{def:retr-smooth}
  A function $f:\mathcal{M}\rightarrow\mathbb{R}$ is said to be  retraction smooth (short to retr-smooth) with constant $L$ and a retraction $\mathcal{R}$,  if for   $\forall~x, y\in\mathcal{M}$  it holds that
  \begin{equation}\label{eq:taylor expansion}
    f(y) \leq f(x) + \left<\grad f(x), \eta\right> + \frac{L}{2}\|\eta\|^2,
  \end{equation}
  where $\eta\in T_x\mathcal{M}$ and $y= \mathcal{R}_x(\eta)$.
\end{definition}

Finally, we show that the retr-smoothness can be induced by the smoothness in usual sense. 
\begin{proposition}[Lemma 2.7, \cite{grocf}]\label{prop:retr}
    Under Assumption \ref{general assumption}, if $f$ is $L$-smooth in the convex hull of $\Mcal$ ($\text{conv}(\Mcal)$), and the gradient $\nabla f(x)$ is bounded by $G$, then $f$ is retr-smooth with constant $\alpha^2 L + 2G\beta$, where $\alpha,\beta$ are given in \eqref{eq:retrac-lipscitz}. 
\end{proposition}

\subsection{Stationary point and smoothing problem}

We start with the KKT condition of problem \eqref{general problem}. In particular, we introduce the auxiliary variable $y = c(x)$:
\begin{equation}
    \min_{x\in\Mcal} f(x) + h(y),~\mathrm{s.t.}~ c(x) = y. 
\end{equation}
The corresponding Lagrangian function is given as follows:
\begin{equation}
     \ell(x,y;z): = f(x) + h(y) + \langle z, c(x) - y \rangle,
\end{equation}
where $z$ is the Lagrangian multiplier. Then the KKT condition of \eqref{general problem} is given as follows \cite{deng2022manifold}: A point $x\in \mathcal{M}$ satisfies the KKT condition if there exists $y,z\in\mathbb{R}^m$ such that
\begin{equation}\label{def:kkt-deter}
    \left\{
\begin{aligned}
    0 &=  \mathcal{P}_{T_x\mathcal{M}}\left(\nabla f(x) + \nabla c(x)^\top z \right), \\
    0 & \in -z+ \partial h(y), \\
   0 & =  c(x)  -y.
\end{aligned}
    \right.
\end{equation}
Using the formulation, we give the $\epsilon$-stationary point of problem \eqref{general problem}.
\begin{definition}\label{def:epsilon}
    We say $x\in \mathcal{M}$ is a $\epsilon$-stationary point of problem \eqref{general problem} if there exists $y\in \mathbb{R}^m$ and $z\in \partial h(y)$ satisfies
    \begin{equation}\label{eq:def:epsilon}
        \mathbb{E} \left[\|\mathcal{P}_{T_x\mathcal{M}}( \nabla  f(x) +  \nabla c(x)^\top z )\| \right] \leq \epsilon,~~ \|c(x) - y\| \leq \epsilon.
    \end{equation}
\end{definition}

When $\epsilon=0$ in Definition \ref{def:epsilon}, it immediately follows that $x$ satisfies the exact KKT condition \eqref{def:kkt-deter}.  Compared with the $\epsilon$-stationarity obtained by directly relaxing the KKT conditions in \cite{deng2024oracle}, our notion of an $\epsilon$-stationary point is more restrictive: we require $z\in\partial h(y)$ exactly, whereas \cite{deng2024oracle} only ensures $\mathrm{dist}(z,\partial h(y))\le\epsilon$.   
Peng et al.\ \cite{peng2022riemannian} define $y\in\mathbb{R}^m$ to be a {\em generalized} $\epsilon$-stationary point if there exists $x\in\Mcal$ such that 
\[
\mathbb{E}\Big[\mathrm{dist}^2\left(0,\; \mathcal{P}_{T_x\Mcal}(\nabla f(y) + \nabla c(x)^\top \partial h(y))\right)\Big]\leq \epsilon,\quad 
\|x-y\|^2\le\epsilon^2.
\]
However, under this definition the point $y$ is not required to lie on the manifold $\Mcal$.  In particular, it is unclear whether the discrepancy $\|c(x)-y\|\le\epsilon$ holds.  In contrast, our definition of $\epsilon$-stationarity inherently enforces $x\in\Mcal$, making it a more appropriate notion in this setting.

Following \cite{peng2022riemannian,beck2023}, we give the smoothed approximation of problem \eqref{general problem} using Moreau envelope in \eqref{def:moreau}. In particular, 
given the smoothing parameter $\mu>0$, we define the smoothed approximation problem:
\begin{equation}\label{smoothing problem}
\min_{x\in\Mcal}F_{\mu}(x): = f(x) + h_{\mu}(c( x)). 
\end{equation}
When $f$ is differentiable, $h$ is convex and $c$ is differentiable, we obtain from Proposition \ref{h proposiztion} that $\nabla F_{\mu}$ can be written as
 \begin{equation}\label{Fmu-euclidean-gradient}
 \begin{aligned}
 \nabla F_{\mu}(x) &= \nabla f(x) + \nabla c(x)^\top \nabla h_{\mu_k}(c(x))  \\
 &= \nabla f(x) +\frac{1}{\mu} \nabla c(x)^\top(c (x) - \prox_{\mu h}(c(x))). 
 \end{aligned}
 \end{equation}
 
In this paper, we solve a sequence of smoothed problems of the form \eqref{smoothing problem}, where the smoothing parameter $\mu$ is dynamically adjusted rather than fixed. The parameter $\mu$ controls the trade-off between smoothness and approximation accuracy. At the early stages of the algorithm, we choose a relatively large value of $\mu$ to ensure favorable smoothness properties. As the iteration proceeds, we gradually decrease $\mu$ so that the smoothed problems more closely approximate the original nonsmooth problem. The detailed algorithm is presented in the next section.

\section{Smoothing algorithmic framework}\label{sec:alg}
Throughout this paper, we make the following assumptions for problem \eqref{general problem}. 
\begin{assumption}\label{general assumption} 
\text{ }
\begin{enumerate}
 [label={\textbf{\Alph*:}}, 
 ref={\theassumption.\Alph*}]
 \item\label{assumption-A}%
 The manifold $\mcM$ is a compact Riemannian submanifold embedded in $\mbR^n$.
 \item\label{assumption-B}  The function $f$ is  $L_f$-Lipschitz coninuous on $\Mcal$ and retr-smooth with a retraction operator $\mathcal{R}$ and constant $L$. The nonlinear mapping $c$ is $L_c$-Lipscitz continuous and  $L_{\nabla c}$-smooth on $\Mcal$.  
  \item\label{assumption-D} For each $\xi\in \mathcal{D}$, $\tilde{f}(x,\xi)$ is differentiable and the Riemannian gradient $\grad  \tilde{f}(x,\xi)$ is unbiased with bounded variance. That is, for all $x\in\mcM$, it satisfies for some $\sigma>0$ that
\begin{equation*}
\begin{aligned}
 & \mbE\left[ \grad  \tilde{f}(x,\xi) \right] = \grad f(x), ~~\mbE\left[ \|\grad  \tilde{f}(x,\xi) - \grad f(x) \|^2\right] \leq \sigma^2.
\end{aligned}
\end{equation*} 
\item\label{assumption-E} The function $\tilde{f}(x,\xi)$ satisfies the average smoothness condition: for any $x\in\Mcal, \zeta\in T_x \Mcal, y = \mathcal{R}_x(\zeta)$, it holds that
  \begin{equation}\label{assum:f}
   \mathbb{E} [\| \grad \tilde{f}(x,\xi) - \mcT_{y}^{x} \grad f(y,\xi)  \|] \leq \tilde{L}\|\zeta\|.
 \end{equation}
     \item\label{assumption-F} The function $F(x) = f(x) + h(c(x))$ has a optimal solution $x_*\in \Mcal$ and the optimal value $F_*$ is finite. 
\end{enumerate}
\end{assumption}

\begin{remark}
We make some remarks on Assumption \ref{general assumption}.  Assumption \ref{assumption-A} includes many common manifolds, such as sphere, Stiefel manifold and Oblique manifold, etc. This implies $\Mcal$ is a bounded and closed set, i.e., there exists a finite constant $D$ such that $D = \max_{x,y\in \Mcal} \| x - y  \|$. Assumption \ref{assumption-B} imply that for any $x,y\in\Mcal$, it holds that
        \begin{equation}
            \begin{aligned}
                \|\nabla c(x)\| \leq L_c,  ~~
                \|\nabla f(x) \| \leq L_f. 
            \end{aligned}
        \end{equation}
        Since $\Mcal$ is compact, it can be easily satisfied.   Assumption \ref{assumption-D} is standard and implies that $\grad \tilde{f}(x,\xi)$ is an unbiased estimator of $\grad f(x)$ and has a bounded variance for any $x\in\Mcal$.   Assumption \ref{assumption-E} is often imposed in the literature to design stochastic algorithms \cite{cutkosky2019momentum,levy2021storm+,lu2024variance,han2020riemannian}.

\end{remark}

We do not impose any assumptions on \( h \) at this point. In the subsequent algorithmic design and convergence analysis, we will separately consider two different assumptions on \( h \). The first scenario assumes $h$ to be Lipschitz continuous, while the second considers $h$ as the indicator function of a convex set $\mathcal{C}$, i.e., $h(x) = \delta_{\mathcal{C}}(x)$.

\subsection{The proposed algorithm for Lipschitz continuous function}\label{sec:lipsch}

This subsection considers the case that $h$ is Lipschitz continuous. We first give the following assumption.
\begin{assumption}\label{assumption-h}
The function $h$ is  convex and $\ell_h$-Lipschitz continuous.  
\end{assumption}

Under this assumption, we propose a stochastic smoothing method with a recursive momentum scheme for solving problem~\eqref{general problem}. Our method builds upon the frameworks of \cite{peng2022riemannian,beck2023}, but introduces a substantially different strategy for constructing the stochastic gradient estimator, leading to improved convergence guarantees over existing stochastic methods \cite{peng2022riemannian,deng2024oracle}.  Starting from an initial point $x_0 \in \Mcal$, we approximately solve a sequence of smoothed subproblems of the form $\min_{x \in \Mcal} F_{\mu_k}(x)$, as defined in \eqref{smoothing problem}, by performing a single stochastic gradient descent step:
\begin{equation}
    x_{k+1} = \mathcal{R}_{x_k}(-\tau_k G_k),
\end{equation}
where $\mathcal{R}_{x_k}$ is a retraction at $x_k$, $\mu_k$ is a smoothing parameter, and $\tau_k$ is an adaptive stepsize. Motivated by \cite{levy2021storm+}, we propose the following update rule:
\begin{equation}\label{eq:constant-cnd}
    \tau_k = \frac{1}{\left(\sum_{i=1}^k \|G_i\|^2 / a_{k+1}\right)^{1/3}}.
\end{equation}
Notice that we use $a_{k+1}$ instead of $a_{i+1}$, which is crucial for the convergence analysis. To approximate the Riemannian gradient $\grad F_{\mu_k}(x_k)$, we exploit its structure:
\[
    \grad F_{\mu_k}(x_k) = \grad f(x_k) + \mathcal{P}_{T_{x_k} \Mcal}\left( \nabla c(x_k)^\top \nabla h_{\mu_k}(c(x_k)) \right),
\]
and define the estimator $G_k$ as:
\[
    G_k = \delta_k + \mathcal{P}_{T_{x_k} \Mcal}\left( \nabla c(x_k)^\top \nabla h_{\mu_k}(c(x_k)) \right),
\]
where $\delta_k$ is a variance-reduced estimator of $\grad f(x_k)$. In particular, we compute $\delta_k$ via a recursive momentum technique inspired by \cite{cutkosky2019momentum}:
\begin{equation}
    \delta_k = \grad \tilde{f}(x_k,\xi_k) + (1 - a_k)\, \mathcal{T}_{x_{k-1}}^{x_k} \left( \delta_{k-1} - \grad \tilde{f}(x_{k-1}, \xi_k) \right),
\end{equation}
where $a_k \in (0,1]$ is the momentum parameter, $\xi_k$ is a randomly drawn sample, and $\mathcal{T}_{x_{k-1}}^{x_k}$ denotes a vector transport operator from $T_{x_{k-1}} \Mcal$ to $T_{x_k} \Mcal$. The detailed process is referred to Algorithm \ref{alg:stomanial}.

\begin{algorithm}[H]
\footnotesize
\caption{A Riemannian stochastic smoothing method with recursive momentum for problem \eqref{general problem}.}\label{alg:stomanial}
\begin{algorithmic}[1]
\REQUIRE Initial point $x_0\in \mathcal{M}$, $\{a_k\}\subset (0,1]$, $\{\mu_k\}_k \subset(0,\infty)$. 
\WHILE {$k =1,2,\cdots,$}
\STATE $G_k = \delta_k + \mathcal{P}_{T_{x_k} \Mcal}( \nabla c(x_k)^\top \nabla h_{\mu_k}(c(x_k)) )$.
\STATE Update $    \tau_k = 1/(\sum_{i=1}^k\|G_i\|^2/a_{k+1})^{1/3}$.
\STATE Update $x_{k+1}=\mathcal{R}_{x_k}(-\tau_k G_k)$.
\STATE  Sample $\xi_{k+1} \sim \mathcal{D}$ and compute the gradient $\grad f(x_{k+1},\xi_{k+1} )$.
\STATE Compute the momentum estimator: 
 $$ \delta_{k+1}= \grad f(x_{k+1}, \xi_{k+1}) + (1 -a_{k+1})\mcT_{x_k}^{x_{k+1}}(\delta_k-\grad \tilde{f}(x_k, \xi_{k+1})).$$
\ENDWHILE
\end{algorithmic}
\end{algorithm}

We highlight several key distinctions between Algorithm~\ref{alg:stomanial} and existing stochastic methods for nonsmooth optimization on compact submanifolds. Most notably, our construction of the stochastic gradient estimator departs fundamentally from those in \cite{deng2024oracle,peng2022riemannian}. In particular, the algorithm in \cite{deng2024oracle} applies a momentum scheme directly to the augmented Lagrangian function $\mathcal{L}_\rho$, thereby introducing uncontrolled variance stemming from its deterministic components (analogous to the smoothed term $h_\mu(x)$ in our Algorithm~\ref{alg:stomanial}). This issue becomes particularly severe when the penalty parameter $\rho$ is large, which impedes effective variance reduction and renders the gradient estimator $G_k$ less reliable. Similarly, the method proposed in \cite{peng2022riemannian} shares this limitation and, in addition, does not incorporate any variance reduction mechanisms. Another major distinction lies in algorithmic structure: both \cite{peng2022riemannian} and \cite{deng2024oracle} rely on double-loop frameworks (although \cite{peng2022riemannian} provides a theoretically weaker single-loop variant), whereas our Algorithm~\ref{alg:stomanial} achieves a fully single-loop implementation without sacrificing convergence guarantees.

These structural differences in gradient estimator design and algorithmic formulation result in significant improvements in complexity bounds. Our method attains the optimal sample complexity of $\mathcal{O}(\epsilon^{-3})$, improving upon the $\mathcal{O}(\epsilon^{-5})$ and $\mathcal{O}(\epsilon^{-3.5})$ rates reported in \cite{peng2022riemannian} and \cite{deng2024oracle}, respectively. Moreover, our analysis extends seamlessly to the case where $h$ is an indicator function, yielding the same optimal complexity bound. This result is consistent with recent findings in Euclidean settings, such as those established in \cite{lu2024variance}, as we detail in the next subsection.

We are now ready to present the convergence results for Algorithm \ref{alg:stomanial}, with the proof deferred to Section \ref{sec:proof-1}. 

\begin{theorem}\label{Theorem of R2SRM}
 Suppose that Assumptions \ref{general assumption} and \ref{assumption-h} hold. Let $\{x_k\}$ be the sequence generated by Algorithm \ref{alg:stomanial}.  If $\mu_k = k^{-1/3}$, $a_1 = 1$ and $a_{k+1}= k^{-2/3}$ for $k\geq 1$.
   Then we have that for given $K>1$
   \begin{equation}\label{eq:bound-grad-F-k}
   \sum_{k=1}^K  \mathbb{E}[\|\grad F_{\mu_k}(x_k)\|^2] \leq (864\sqrt{3} \tilde{L}^3 + 96\sigma^2 + 64 B^{3/2} + 9^4 \mathcal{G}^3)K^{1/3},
\end{equation}
where $B=(L_f + L_c)D + \frac{5}{2} \ell_h^2 \mu_1$ and $\mathcal{G}$ is given in Lemma \ref{Euclidean vs manifold1}. 
\end{theorem}

The following result is an immediate consequence of Theorem \ref{Theorem of R2SRM}. It provides iteration complexity
results for Algorithm \ref{alg:stomanial} to find an $\epsilon$-stochastic stationary point $x_k$ of problem \eqref{general problem} satisfying \eqref{eq:def:epsilon}.

\begin{corollary}
Under the same setting as in Theorem \ref{Theorem of R2SRM},   let $i_K$ be the random variable uniformly generated from $\{\lceil K/2 \rceil,\cdots,K\}$.  Then we have for any $\epsilon>0$, there exists some $K = {\mathcal{O}}(\epsilon^{-3})$ such that
    \begin{equation}
        \mathbb{E}\left[\|\mathcal{P}_{T_{x_{i_K}}\Mcal}(\nabla f(x_{i_K}) +  \nabla c(x_{i_K})^\top z_{i_K})\|\right]\leq \epsilon,~~ \|c(x_{i_K}) - y_{i_K}\| \leq \epsilon,
    \end{equation}
    where $y_{i_K} = \prox_{\mu_{i_K} h}(c(x_{i_K}))$ and $z_{i_K} = \frac{1}{\mu_{i_K}} (c(x_{i_K})  - y_{i_K}) \in \partial h(y_{i_K})$. 
    \end{corollary}

\begin{proof}

It follows from \eqref{h bound} and  $y_k = \prox_{\mu_k h}(c(x_k))$ that
 \begin{equation}\label{bound-cx-y}
     \|c(x_k) - y_k\|^2 \leq \mu_k^2 \ell_h^2. 
 \end{equation}
Substituting $\mu_k = k^{-1/3}$ and  applying the definition of $i_K$ yields
 \begin{equation}
     \begin{aligned}
       \|c(x_{i_K}) - y_{i_K}\|  \leq (i_K)^{-1/3} \ell_h \leq (\lceil K/2\rceil)^{-1/3} \ell_h \leq 2 \ell_h K^{-1/3}.
     \end{aligned}
 \end{equation}
Using the bound \eqref{eq:bound-grad-F-k} and applying the definition of $i_K$ and Jensen’s inequality gives
\begin{equation}
\begin{aligned}
   \mathbb{E}\left[ \| \grad F_{\mu_{i_K}}(x_{i_K}) \|\right] 
   &\leq \sqrt{ \mathbb{E}\left[   \| \grad F_{\mu_{i_K}}(x_{i_K}) \|^2\right]} \\
   &= \sqrt{  \frac{1}{K - \lceil K/2 \rceil +1} \sum_{k=\lceil K/2 \rceil }^K  \mathbb{E}  \| \grad F_{\mu_{k}}(x_{k}) \|^2  }\\
   & = \sqrt{2 (863\sqrt{3} \tilde{L}^3 + 96\sigma^2 + 64 B^{3/2} + 9^4 \mathcal{G}^3)} K^{-1/3}.
   \end{aligned}
\end{equation}
% \begin{equation}
% \begin{aligned}
%     \mathbb{E}\left[\| \grad F_{\mu_{i_K}}(x_{i_K}) \|\right] + \|c_{x_{i_K}} - y_{i_K}\| & 
%     \leq   \sqrt{ 2(  \mathbb{E}[\| \grad F_{\mu_K}(\bar{x}_K) \|^2]  + \|x_{{i_K}} - y_{i_K}\|^2) }  \\
%     &= \sqrt{ 2\sum_{k=1}^K \frac{ \mathbb{E}[ \| \grad F_{\mu_k}(x_k) \|^2] + \|c(x_k) - y_k\|}{K} } \\
%     &\leq \mathcal{O}( \sqrt{\ell_h^2 + L^3 + \sigma^2 + B^{3/2}}  K^{-1/3}).
%     \end{aligned}
% \end{equation}
Moreover, by the definition of $z_{i_K}$, we obtain
\begin{equation}
\begin{aligned}
    \grad F_{\mu_{i_K}}(x_{i_K}) &= \mathcal{P}_{T_{x_{i_K}}\Mcal}(\nabla f(x_{i_K}) + \frac{1}{\mu_{i_K}} \nabla c(x_{i_K})^\top (c(x_{i_K})  - y_{i_K})) \\
    &= \mathcal{P}_{T_{x_{i_K}}\Mcal}(\nabla f(x_{i_K}) +  \nabla c(x_{i_K})^\top z_{i_K}).
    \end{aligned}
\end{equation}
This completes the proof.
\end{proof}

\begin{remark}
Since Algorithm \ref{alg:stomanial} is a single-loop algorithm, and it requires one sample, one gradient evaluation of $c$, and two gradient evaluations of $\tilde{f}$ per iteration, its sample and first-order operation complexity are of the
same order as its iteration complexity. Additionally, our algorithm is a parameter-free momentum based method that ensures the optimal iteration complexity of $\mathcal{O}(\epsilon^{-3})$. Consider the smooth cases, i,e., $h = 0$, our algorithm improve over the Riemannian stochastic recursive momentum method \cite{han2020riemannian} by shaving off a $\log(K)$ factor, which is due to we use the improved momentum technique from \cite{levy2021storm+}.

\end{remark}

\subsection{The proposed algorithm for indicator function}\label{sec:indicator}
This subsection considers the case that $h$ is an indicator function over some convex set $\mathcal{C}$.  In particular, we make the following assumption.
\begin{assumption}\label{assum:h-2}
    The function $h$ in \eqref{general problem} is an indicator function over a convex set $\mathcal{C}$, i.e., $h(x) = \delta_{\mathcal{C}}(x)$. 
\end{assumption}

Under Assumption \ref{assum:h-2},  problem \eqref{general problem} reduces to the following smooth optimization on compact submanifold with extra constraint:
\begin{equation}\label{prob:constrint}
    \min_{x\in \Mcal} f(x),~\mathrm{s.t.},~c(x)\in \mathcal{C}. 
\end{equation}
The proximal operator $\prox_{h}(x)$ reduces to the projection operator $\mathcal{P}_{\mathcal{C}}(x)$. Similarly, the Moreau envelope simplifies to the scaled squared distance function $\frac{1}{2\mu}\text{dist}^2(x,\mathcal{C})$. Consequently, the corresponding smooth approximation of $F$ is given as follows:
\begin{equation}
    F_{\mu}(x) = f(x) + \frac{1}{2\mu}\text{dist}^2(c(x),\mathcal{C}).
\end{equation}
This is equivalent to a sequence of quadratic penalty functions and $1/\mu$ is the penalty parameter. Finally, one can obtains the Euclidean gradient of $F_{\mu}(x)$:
\begin{equation}
    \nabla F_{\mu}(x):= \nabla f(x) + \frac{1}{\mu} \nabla c(x)^\top\left[c(x) - \mathcal{P}_{\mathcal{C}}(c(x))\right].
\end{equation}

Since that $\partial h(x) = N_{\mathcal{C}}(x)$ for any $x\in \mathcal{C}$. Following Definition \eqref{def:epsilon}, we give the definition of $\epsilon$-stationary point for \eqref{prob:constrint}. 

\begin{definition}\label{def:epsilon-1}
    We say $x\in \mathcal{M}$ is a $\epsilon$-stationary point of problem \eqref{prob:constrint} if there exists $y\in \mathcal{C}$ and $z\in N_{\mathcal{C}}(y) $ satisfies
    \begin{equation}\label{eq:def:epsilon-1}
        \mathbb{E}[\|\mathcal{P}_{T_x\mathcal{M}}( \nabla  f(x) +  \nabla c(x)^\top z )\|] \leq \epsilon,~~  \mathrm{dist}(c(x),\mathcal{C}) \leq \epsilon.
    \end{equation}
\end{definition}

Following Algorithm~\ref{alg:stomanial}, we present Algorithm~\ref{alg:stomanial-1} for solving problem~\eqref{general problem} under Assumption~\ref{assum:h-2}. This algorithm shares the same structure as Algorithm~\ref{alg:stomanial}, except for the update rule of the estimator $\delta_k$. Specifically, we employ the truncated recursive momentum scheme to compute $\delta_k$ as
\begin{equation}\label{def:delta:indicator}
    \delta_k = \mathcal{P}_{\mathcal{B}_{L_f}}\left( \grad \tilde{f}(x_k,\xi_k) + (1-a_{k-1})\mcT_{x_{k-1}}^{x_k}(\delta_{k-1} - \grad \tilde{f}(x_{k-1},\xi_k)) \right),
\end{equation}
where $a_{k-1} \in (0,1]$ is the momentum parameter, $\xi_k$ is a randomly drawn sample, and $\mathcal{P}_{\mathcal{B}_{L_f}}$ denotes the projection onto the ball of radius $L_f$ (centered at the origin) in the tangent space. This projection ensures that $\delta_k$ remains bounded, which is crucial for controlling the variance of the estimator in the convergence analysis. Note that the algorithm does not require the exact value of $L_f$ in advance. In fact, without loss of generality, $L_f$ may be replaced by any constant larger than the true Lipschitz constant. For notational simplicity, we use $L_f$ throughout.

\begin{algorithm}[H]
\footnotesize
\caption{A Riemannian stochastic smoothing method with truncated recursive momentum for problem \eqref{prob:constrint}.}\label{alg:stomanial-1}
\begin{algorithmic}[1]
\REQUIRE Initial point $x_0\in \mathcal{M}$, $\{\mu_k\}_k, \{ a_k \}_k,\{ \tau_k \}_k $. Set $k=0$. 
\WHILE {$k =0, 1,2,\cdots,$}
\STATE $G_k = \delta_k + \mathcal{P}_{T_{x_k} \Mcal}( \nabla c(x_k)^\top \nabla h_{\mu_k}(c(x_k)) )$.
\STATE Update $x_{k+1}=\mathcal{R}_{x_k}(-\tau_k G_k)$.
\STATE  Sample $\xi_{k+1} \sim \mathcal{D}$ and compute the gradient $\grad f(x_{k+1},\xi_{k+1} )$.
\STATE Compute the momentum estimator: 
 $$ \delta_{k+1}= \mathcal{P}_{\mathcal{B}_{L_f}}\left(\grad f(x_{k+1}, \xi_{k+1}) + (1 -a_{k+1})\mcT_{x_k}^{x_{k+1}}(\delta_k-\grad \tilde{f}(x_k, \xi_{k+1}))\right).$$
\ENDWHILE
\end{algorithmic}
\end{algorithm}

Since the indicator function associated with the constraint set $\mathcal{C}$ is not Lipschitz continuous, a regularity condition is often required to facilitate algorithmic analysis. Following standard practice in constrained optimization, we introduce a distance function
\[\label{def:g}
g(x) := \frac{1}{2} \mathrm{dist}^2(c(x), \mathcal{C}).
\]

\begin{assumption}\label{assum:error-bound}
There exist constants $\zeta > 0$ and $\theta \geq 1$ such that the following inequality holds for all $x$:
\begin{equation}\label{eq:error-bound}
    \| \grad g(x) \| \geq \zeta \, \mathrm{dist}^{\theta}(c(x), \mathcal{C}).
\end{equation}
\end{assumption}

Assumption~\ref{assum:error-bound} generalizes the commonly used error bound condition with $\theta = 1$, which has been extensively studied in the context of nonconvex constrained optimization (e.g., \cite{sahin2019inexact,li2021rate,li2024stochastic,lu2024variance,alacaoglu2024complexity}).  To interpret \eqref{eq:error-bound}, we note that the Euclidean gradient of $g$ is given by
$
\nabla g(x) = \nabla c(x)^\top \left( c(x) - \mathcal{P}_{\mathcal{C}}(c(x)) \right)$, 
and its Riemannian counterpart is
\[
\begin{aligned}
\grad g(x) &= \mathcal{P}_{T_x \Mcal} \left( \nabla c(x)^\top \left( c(x) - \mathcal{P}_{\mathcal{C}}(c(x)) \right) \right) \\
&= \nabla c(x)^\top \left( c(x) - \mathcal{P}_{\mathcal{C}}(c(x)) \right) - \mathcal{P}_{N_x \Mcal} \left( \nabla c(x)^\top \left( c(x) - \mathcal{P}_{\mathcal{C}}(c(x)) \right) \right).
\end{aligned}
\]
Hence, Assumption~\ref{assum:error-bound} implies the following geometric inequality:
\[
\mathrm{dist} \left( \nabla c(x)^\top \left( c(x) - \mathcal{P}_{\mathcal{C}}(c(x)) \right), -N_x \Mcal \right) \geq \zeta \, \mathrm{dist}^{\theta}(c(x), \mathcal{C}).
\]
When $\mathcal{C} = \{ y \in \mathbb{R}^m : y = 0 \}$ and $\Mcal = \mathcal{X}$ is a convex set in Euclidean space, this inequality recovers the classical error bound condition in \cite{lu2024variance}. 

Now we give main result of Algorithm \ref{alg:stomanial} under Assumptions \ref{assum:h-2} and \ref{assum:error-bound}.

\begin{theorem}\label{Theorem of R2SRM-1}
  Suppose that Assumptions \ref{general assumption}, \ref{assum:h-2} and \ref{assum:error-bound} hold. Let $\{x_k\}$ be the sequence generated by Algorithm \ref{alg:stomanial-1}. If $\mu_k = k^{-\omega}$, $\tau_k=c_{\tau}(k+1)^{-\omega}$, $a_k=c_{a} k^{-2\omega}$, where $\omega = \min\{ \frac{\theta}{\theta+2}, \frac{1}{2} \}$, $c_{\tau}$ and $c_a$ satisfy 
    \begin{equation}
        c_{\tau} \leq \min\{\frac{1}{L_g}, \frac{1}{\mathcal{G}} \},~~ c_a = c_{\tau}^2 ( \frac{1}{2} +  \frac{1}{32c_{\tau}^2 \tilde{L}^2} + 4C\tilde{L}^2),
    \end{equation}
    where $L_g$ is defined in \eqref{def:L} and $\mathcal{G}$ is defined in Lemma \ref{Euclidean vs manifold1}. Denote $\tilde{K} := 
 \lceil \frac{8\omega}{c_{\tau} \zeta^2}  \rceil$. 
 Then we have that for given $K>2\tilde{K}$,
\begin{equation}\label{eq:temp20-1}
\begin{aligned}
        & \sum_{k=\tilde{K}}^K \tau_k \mathbb{E}\| \grad F_{\mu_k}(x_k)\|^{2}  \\
     \leq &  4(\Phi_{\tilde{K}} - \Phi_{*}) + (4C_g\omega + \frac{4\sigma^2 c_a^2}{8\tilde{L}^2 c_{\tau}}) \max\{  K^{\omega -\frac{2\omega}{\theta}},  1\}   \ln(K),
\end{aligned}
\end{equation}
 where  $C_g$ is given in Lemma \ref{eq:cond-mu}, $\Phi_k = \mathbb{E}[F_{\mu_k}(x_k)] +\frac{\|\delta_k - \grad f(x_k)\|}{16\tau_0 \tilde{L}^2}$ and $\Phi_* = F_*$.   
\end{theorem}

The following result is an immediate consequence of Theorem \ref{Theorem of R2SRM-1}.

\begin{corollary}
Under the same setting as in Theorem \ref{Theorem of R2SRM-1},   let $i_K$ be the random variable satisfy that $Prob(i_K = k) = \frac{\tau_k}{\sum_{i=\lceil K/2 \rceil}^K \tau_i}$ from $\{\lceil K/2 \rceil,\cdots,K\}$, where $K > 2\tilde{K}$. Let $y_k = \mathcal{P}_{\mathcal{C}}(c(x_{k}))$ and $z_{k} = \frac{1}{\mu_{k}} (c(x_{k})  - \mathcal{P}_{\mathcal{C}}(c(x_{k})))$.  Then we have for any $\epsilon>0$, there exists some $K = \tilde{\mathcal{O}}(\epsilon^{-\max\{\theta+2,2\theta\}})$ such that
    \begin{equation}
        \mathbb{E}\left[\|\mathcal{P}_{T_{x_{i_K}}\Mcal}(\nabla f(x_{i_K}) +  \nabla c(x_{i_K})^\top z_{i_K})\|\right]\leq \epsilon,~~ \mathrm{dist}(c(x_{i_K}),\mathcal{C}) \leq \epsilon.
    \end{equation} 
    \end{corollary}

    \begin{proof}
        By the definition of $y_k$ and $z_k$, it holds that $z_k\in N_{\mathcal{C}}(y_k)$ and  
\begin{equation}\label{eq:gradF-indicator}
\begin{aligned}
\grad F_{\mu_k}(x_k)&  = \mathcal{P}_{T_{x_k}\Mcal} \left(  \nabla f(x_k) + \frac{1}{\mu_k} \nabla c(x_k)^\top ( c(x_k) - \mathcal{P}_{\mathcal{C}}(c(x_k))  ) \right) \\
& = \mathcal{P}_{T_{x_k}\Mcal} \left( \nabla  f(x_k) +  \nabla c(x_k)^\top z_k \right).
\end{aligned}
\end{equation}
Let us define the constant $C_1: = (4C_g\omega +\frac{4\sigma^2 c_a^2}{8\tilde{L}^2 c_{\tau}})$.  Plugging \eqref{eq:gradF-indicator} into \eqref{eq:temp20-1} and combining with the definition of $x_{i_K}$, we obtain
\begin{equation}\label{eq:gradF-output-bound-1}
 \begin{aligned}
     &\mathbb{E}\left[\left\|  \mathcal{P}_{T_{x_{i_K}}\mathcal{M}}(\nabla f(x_{i_K})+  \nabla c(x_{i_K})^\top z_{i_K} )\right\|^2 \right]  
     =  \mathbb{E}[ \|\grad F_{\mu_{i_K}}(x_{i_K})\|^2 ]\\
     = & \frac{\sum_{i=\lceil K/2 \rceil}^K \tau_i \mathbb{E}[ \|\grad F_{\mu_{i}}(x_{i})\|^2 ]  }{ \sum_{i=\lceil K/2 \rceil}^K \tau_i }   \\
      %\leq & \frac{\sum_{i=1}^K \tau_i \mathbb{E}[ \|\grad F_{\mu_{i}}(x_{i})\|^2 ]  }{ \sum_{i=\lceil K/2 \rceil}^K \tau_i  }   \\
     \leq & \frac{4(\Phi_{\tilde{K}} - \Phi_{*}) + C_1 \max\{  K^{\omega -\frac{2\omega}{\theta}},  1\}   \ln(K) }{\sum_{i=\lceil K/2 \rceil}^K \tau_i } \\
     \leq &  \frac{4C_1}{3c_{\tau}} \max\left\{\frac{1}{K^{1-2\omega+\frac{2\omega}{\theta}}},   \frac{1}{K^{1-\omega}}\right\} \ln(K) + \frac{16(\Phi_{\tilde{K}} - \Phi_{*}) }{3c_\tau K^{1-\omega}},
 \end{aligned} 
\end{equation}
where the last inequality uses the fact that $\tau_k = c_{\tau}k^{-\omega}$ and 
        \begin{equation*}
        \begin{aligned}
            \sum_{k=\lceil K/2 \rceil}^K k^{-\omega} & \geq \sum_{k=\lceil K/2 \rceil}^K \int_k^{k+1} x^{-\omega} \mathrm{~d} x \geq \frac{1}{1-\omega} ((K+1)^{1-\omega} - (\lceil K/2 \rceil)^{1-\omega})\\
            &\geq \frac{1}{1-\omega} (1-\frac{1}{2^{1-\omega}}) K^{1-\omega} \geq \frac{3}{4}  K^{1-\omega},
            \end{aligned}
        \end{equation*}
        where the last inequality uses $\frac{1}{3}\leq \omega \leq \frac{1}{2}$. 
         On the other hand, it follows from \eqref{eq:bound:g} that
   \begin{equation}
       \text{dist}^2(c(x_{i_K}), \mathcal{C}) \leq 2C_g  (i_K)^{-\frac{2\omega}{\theta} } \leq 2^{1-2\omega/\theta}C_g    K^{-\frac{2\omega}{\theta} }.
   \end{equation}
       Since $\omega = \min\{ \frac{\theta}{\theta +2}, \frac{1}{2} \}$ and $\theta \geq 1$, it is straightforward to verify that all the exponents $1-2\omega + \frac{2\omega}{\theta}, \frac{2\omega}{\theta},1-\omega$ are strictly positive. Furthermore, we have
       $$
       \min\{ 1-2\omega + \frac{2\omega}{\theta}, \frac{2\omega}{\theta},1-\omega \} = \min\{ \frac{2}{\theta+2}, \frac{1}{\theta} \}.
       $$
       Applying Jensen’s inequality finally completes the proof.
       % $$
       %  \mathbb{E}\left[\|\mathcal{P}_{T_{x_{i_K}}\Mcal}(\nabla f(x_{i_K}) +  \nabla c(x_{i_K})^\top z_{i_K})\|\right]+ \mathrm{dist}(c(i_K),\mathcal{C}) \leq \tilde{\mathcal{O}}(K^{-\min\{ \frac{1}{\theta+2}, \frac{1}{2\theta} \}}).
       % $$
       % The proof is completed.
    \end{proof}

\section{Proof of main result}

\subsection{Common lemmas}

We begin with two technical lemmas that are useful for the subsequent analysis.

\begin{lemma}[\cite{levy2021storm+}, Lemma 3]\label{lem:bound-bk}
    Let $b_1>0,b_2,\cdots,b_n\geq 0$ be a sequence of real numbers, $p\in(0,1)$ be a real number. The following inequality holds
    \begin{equation}\label{eq:bound-bk}
        \sum_{k=1}^n \frac{b_k}{(\sum_{i=1}^k b_i)^p} \leq \frac{1}{1-p} \left(  \sum_{k=1}^n b_k  \right)^{1-p}.
    \end{equation}
\end{lemma}

\begin{lemma}\label{lem:bound-x}
 Let $c,d>0$ be two real numbers. Suppose that there exist $\alpha, \beta \in  (0,1)$ such that $x\leq c x^{\alpha} + d x^{\beta} + e$. The following inequality holds
 \begin{equation}\label{eq:bound-x}
     x\leq 2(4\alpha)^{\frac{\alpha}{1-\alpha}}c^{\frac{1}{1-\alpha}} + 2(4\beta)^{\frac{\beta}{1-\beta}}d^{\frac{1}{1-\beta}}+ 2e.
 \end{equation}
\end{lemma}

\begin{proof}
    Applying Young's inequality to $cx^{\alpha}$ yields
    \begin{equation}\label{eq:bound-c}
        cx^{\alpha} =  \frac{1}{4\alpha} (4\alpha c *  x^{\alpha}) \leq \frac{1}{4\alpha} ( (4\alpha c)^{\frac{1}{1-\alpha}} + \alpha x ) = \frac{1}{4} x + (4\alpha)^{\frac{\alpha}{1-\alpha}}c^{\frac{1}{1-\alpha}}.
    \end{equation}
    Similarly, one has that $dx^{\beta} \leq  \frac{1}{4} x + (4\beta)^{\frac{\beta}{1-\beta}}d^{\frac{1}{1-\beta}}$. 
    Combining with \eqref{eq:bound-c} yields
    \begin{equation}
        x\leq \frac{1}{2} x + (4\alpha)^{\frac{\alpha}{1-\alpha}}c^{\frac{1}{1-\alpha}} + (4\beta)^{\frac{\beta}{1-\beta}}d^{\frac{1}{1-\beta}}+ e.
    \end{equation}
    Re-ranging the inequality completes the proof.
\end{proof}

To apply the Riemannian stochastic gradient method to the smoothed problem \eqref{smoothing problem}, it is necessary to verify that $F_\mu$ satisfies the retr-smoothness condition defined in Definition~\ref{def:retr-smooth}. The following lemma establishes this property under two cases of $h$.

\begin{lemma}\label{Euclidean vs manifold1} 
  Suppose that Assumption \ref{general assumption} holds. If $h$ is Lipschitz continuous with constant $\ell_h$ or is the indicator function of a convex set $\mathcal{C}$, then for any $\mu \in (0,1)$, and $x\in\mathcal{M}$, and $F_{\mu}$ is retr-smooth in the sense that
  \begin{equation}\label{equ:L2111}
    F_{\mu}(\mathcal{R}_{x}(\eta)) \leq F_{\mu}(x) + \left<\eta,\grad F_{\mu}(x)\right> + \frac{\mathcal{G}}{2\mu}\|\eta\|^2
  \end{equation}
  for all $\eta\in T_{x}\mathcal{M}$, where $\mathcal{G} = \max\{\mathcal{G}_1,\mathcal{G}_2\}$ and $\mathcal{G}_1, \mathcal{G}_2$ are given in the proof. 
\end{lemma}
\begin{proof}

Since $F_{\mu}(x) = f(x) + h_{\mu}(c(x))$ and $f$ is retr-smooth with constant $L$ by Assumption~\ref{assumption-B}, it suffices to analyze the retr-smoothness of the composite term $h_{\mu}(c(x))$. To this end, we define $s(x) := h_{\mu}(c(x))$, and examine the following two cases separately.

  \textbf{Case 1: $h$ is Lipschitz continuous with $\ell_h$.} It follows from the formula of $\nabla h_{\mu}$ that $\|\nabla s(x)\| \leq  L_c \ell_h$ for all $x\in\mathcal{M}$.
Additionally, it follows from Proposition \ref{h proposiztion} that
$$
\begin{aligned}
& \|\nabla s(x) - \nabla s(y)\| \\
\leq & \| \nabla c(x)^\top (\nabla h_{\mu_k}(c(x)) - \nabla h_{\mu_k}(c(y)))\| + \| (\nabla c(x) - \nabla c(y))^\top \nabla h_{\mu_k}(c(y)) \| \\
\leq &  (\frac{L_c^2 }{\mu} +  \ell_h L_{\nabla c}) \|x - y\|.
\end{aligned}
$$
Now we use Proposition \ref{prop:retr} to obtain $s$ is retr-smooth with constant $\alpha^2(\frac{L_c^2 }{\mu} +  \ell_h L_{\nabla c})+ 2L_c \ell_h \beta$. Since $\mu\leq 1$, we set $\mathcal{G}_1 = L + \alpha^2(L_c^2 + \ell_h L_{\nabla c}) +2 L_c\ell_h \beta$. Therefore, $F_{\mu}$ is retr-smooth with constant $\mathcal{G}_1$ in this case.

\textbf{Case 2: $h(x) = \delta_{\mathcal{C}}(x)$ with a convex set $\mathcal{C}$.} 
  Since that $\mathrm{dist}(c(x),\mathcal{C}) = \|c(x) - \mathcal{P}_{\mathcal{C}}(c(x))\|$ is a continuous function and $\Mcal$ is compact submanifold, there exists a constant $M$ such that $\|c(x) - \mathcal{P}_{\mathcal{C}}(c(x))\| \leq M$ for any $x\in \Mcal$.  Then we have that  $\|\nabla s(x)\| \leq  \frac{L_c M}{\mu}$. Additionally, it follows from Proposition \ref{h proposiztion} that
  $$
\begin{aligned}
& \|\nabla s(x) - \nabla s(y)\| \\
\leq & \| \nabla c(x)^\top (\nabla h_{\mu_k}(c(x)) - \nabla h_{\mu_k}(c(y)))\| + \| (\nabla c(x) - \nabla c(y))^\top \nabla h_{\mu_k}(c(y)) \| \\
\leq &  (\frac{L_c^2 }{\mu} +\frac{ M L_{\nabla c}}{\mu}) \|x - y\|.
\end{aligned}
$$
  It follows from Lemma 2.7 in \cite{grocf} that $s$ is retr-smooth with constant $\frac{ \alpha^2 (  L_c^2 + M L_{\nabla c})}{\mu}  + 2 \frac{L_c M}{\mu}\beta$. We set $\mathcal{G}_2 = L + \alpha^2(L_c^2 + M L_{\nabla c}) + 2L_c M \beta$. Therefore, $F_{\mu}$ is retr-smooth with constant $\mathcal{G}_2$ in this case.

  This completes the proof by considering both cases.
\end{proof}

In Lemma~\ref{Euclidean vs manifold1}, we assume that the smoothing parameter satisfies $\mu \leq 1$, which in fact holds throughout the subsequent analysis.

\subsection{Proof of Section \ref{sec:lipsch}}\label{sec:proof-1}

Let us denote 
\begin{equation}
    \epsilon_k = \delta_k - \grad f(x_k).
\end{equation}
We first present a Lemma that bounds the estimation error of the recursive momentum estimator.  
\begin{lemma}[Estimation error bound]\label{Estimation error bound}
Suppose Assumptions \ref{general assumption} holds and consider Algorithm \ref{alg:stomanial}. Then the expected estimation error of the estimator is bounded by
\begin{equation}\label{eq:recursive:epsilon}
    \mathbb{E}[ \|\epsilon_k\|^2] \leq   (1-a_k)\mathbb{E}[\|\epsilon_{k-1}\|^2]+2a_k^2\sigma^2+2\tilde{L}^2\tau_{k-1}^2\mathbb{E}[\|G_{k-1}\|^2].
\end{equation}
\end{lemma}

\begin{proof}
Let $\mathcal{F}_k = \{\xi_1,\cdots,\xi_{k-1}\}$ denote the collection of samples drawn up to iteration $k-1$ in Algorithm \ref{alg:stomanial}.  
It follows from the update rule of $\delta_k$ that
$$
\begin{aligned}
     & \delta_k- \grad f(x_k) \\
    = &  (1-a_k)\mcT_{x_{k-1}}^{x_k}(\delta_{k-1}- \grad \tilde{f}(x_{k-1}, \xi_k))+ \grad \tilde{f}(x_k, \xi_k)- \grad f(x_k) \\
     = &  (1-a_k)\mcT_{x_{k-1}}^{x_k}(\delta_{k-1}- \grad f(x_{k-1}))+a_k( \grad \tilde{f}(x_k, \xi_k)- \grad f(x_k)) \\
   & +(1-a_k)\underbrace{( \grad \tilde{f}(x_k, \xi_k)-\mcT_{x_{k-1}}^{x_k} \grad \tilde{f}(x_{k-1}, \xi_k )+\mcT_{x_{k-1}}^{x_k} \grad f(x_{k-1})- \grad f(x_k))}_{Z_k}.
\end{aligned}
$$
It follows from Assumption \ref{assumption-D} that
\begin{equation}\label{eq:zero-expect}
\begin{aligned}
   & \mathbb{E}[ \langle 
 \mcT_{x_{k-1}}^{x_k}(\delta_{k-1}- \grad f(x_{k-1})), ( \grad \tilde{f}(x_k, \xi_k)- \grad f(x_k)) \rangle \vert \mathcal{F}_k  ] = 0,\\
& \mathbb{E}[ \langle 
 \mcT_{x_{k-1}}^{x_k}(\delta_{k-1}- \grad f(x_{k-1})), Z_k \rangle \vert \mathcal{F}_k  ] = 0.
 \end{aligned}
\end{equation}
Then we have that
\begin{equation}
 \begin{aligned}
&\mbE[\|\epsilon_k\|^2 \lvert \mcF_k]\\
=&(1-a_k)^2\|\epsilon_{k-1}\|^2+\mbE[\|a_k (\grad \tilde{f}(x_k, \xi_k)- \grad f(x_k))+(1-a_k)Z_k\|^2\lvert \mcF_k] \\
\leq &(1-a_k)^2\|\epsilon_{k-1}\|^2+2a_k^2\mbE[\| \grad \tilde{f}(x_k, \xi_k)- \grad f(x_k)\|^2\lvert \mcF_k] +2(1-a_k)^2\mbE[\| Z_k\|^2\lvert \mcF_k] \\
\leq &(1-a_k)^2\|\epsilon_{k-1}\|^2+2a_k^2\mbE[\| \grad \tilde{f}(x_k, \xi_k)- \grad f(x_k)\|^2\lvert \mcF_k]\\
&+2(1-a_k)^2\mbE[\| \grad \tilde{f}(x_k, \xi_k)-\mcT_{x_{k-1}}^{x_k} \grad \tilde{f}(x_{k-1}, \xi_k )\|^2\lvert \mcF_k] \\
\leq &(1-a_k)^2\|\epsilon_{k-1}\|^2+2a_k^2\sigma^2+2(1-a_k)^2\tau_{k-1}^2\tilde{L}^2\|G_{k-1}\|^2,
\end{aligned}
\end{equation}
where the first equality uses isometry property of vector transport $\mcT_{x_{k-1}}^{x_k}$ and \eqref{eq:zero-expect}, the first inequality follows from  $\|a+b\|^2\leq2\|a\|^2+2\|b\|^2$. The second inequality applies $\mbE\|x-\mbE[x]\|^2\leq\mbE\|x\|^2$ for random variable $x$. The last inequality follows from Assumptions \ref{assumption-D}. Taking full expectation on both sides of the inequality, we obtain \eqref{eq:recursive:epsilon} and complete the proof.

\end{proof}

The following lemma establish the relationship between the  error estimation and the gradient estimation. 
\begin{lemma}\label{Estimation error bound-1}
Suppose Assumptions \ref{general assumption} holds and consider Algorithm \ref{alg:stomanial}. Then the expected estimation error of the estimator is bounded by
\begin{equation}\label{eq:bound:epsilon-k}
\begin{aligned}
\mathbb{E}\sum_{k=1}^K \|\epsilon_k\|^2 \leq 18\tilde{L}^2 K^{2/9} \left( \mathbb{E}\sum_{k=1}^K \|G_k\|^2\right)^{1/3} + 24 \sigma^2 K^{1/3}. 
\end{aligned}
\end{equation}
\end{lemma}

\begin{proof}

Dividing \eqref{eq:recursive:epsilon} by $a_k$ and summing over $k$ yields
\begin{equation}
 \sum_{k=1}^K \frac{1}{a_k}  \mathbb{E}[ \|\epsilon_k\|^2] \leq  \sum_{k=1}^K  (\frac{1}{a_k}-1)\mathbb{E}[\|\epsilon_{k-1}\|^2]+2\tilde{L}^2\sum_{k=1}^K \frac{\tau_{k-1}^2}{a_k}\mathbb{E}[\|G_{k-1}\|^2]+2\sigma^2 \sum_{k=1}^Ka_k.
\end{equation}
Rearranging terms results in
\begin{equation}\label{eq:bound-epslion}
\begin{aligned}
    \mathbb{E} \sum_{k=1}^K \|\epsilon_{k-1}\|^2 \leq & - \frac{\mathbb{E}[\|\epsilon_{T}\|^2]}{a_T} + \underbrace{\sum_{k=1}^{K-1}(\frac{1}{a_{k+1}} - \frac{1}{a_k})\mathbb{E}[\|\epsilon_k\|^2]}_{c_1} \\
    &+ 2\tilde{L}^2 \underbrace{\mathbb{E}[ \sum_{k=1}^K \tau_{k-1}^2 \|G_{k-1}\|^2/a_k ]}_{c_2} + 2\sigma^2 \sum_{k=1}^K a_k.
    \end{aligned}
\end{equation}
Let us bound $c_1,c_2$, individually.  For $c_1$,  since $a_{k+1} =  k^{-2/3}$, one has that for any $k\geq 2$,  $\frac{1}{a_{k+1}} - \frac{1}{a_k} = k^{2/3} - (k-1)^{2/3} \leq \frac{2}{3}(k-1)^{-1/3} \leq \frac{2}{3},$ which is due to the concavity of the function $t^{2/3}$. Moreover, $\frac{1}{a_2} - \frac{1}{a_1} = 0$.
Hence, we have
\begin{equation}
    \sum_{k=1}^{K-1}(\frac{1}{a_{k+1}} - \frac{1}{a_k})\mathbb{E}[\|\epsilon_k\|^2] \leq \frac{2}{3}\sum_{k=1}^K\mathbb{E}[\|\epsilon_k\|^2].
\end{equation}
For $c_2$, it follows from $ a_k \leq a_{k-1} \leq \cdots \leq 1$ that $$\tau_k = \frac{1}{(\sum_{i=1}^k\|G_i\|^2/a_{k+1})^{1/3}} \leq \frac{1}{(\sum_{i=1}^k\|G_i\|^2/a_{i+1})^{1/3}}.$$ 
Plugging this into $c_2$ yields
\begin{equation}
    \begin{aligned}
        & \mathbb{E}[ \sum_{k=1}^K \tau_{k-1}^2 \|G_{k-1}\|^2/a_k ]  \leq \mathbb{E} \sum_{k=1}^K \frac{\|G_{k-1}\|^2/a_k}{(\sum_{i=1}^{k-1}\|G_i\|^2/a_{i+1})^{2/3}} \\
        \leq & 3\mathbb{E}\left( \sum_{k=1}^{K-1} \|G_k\|^2/a_{k+1} \right)^{1/3}  \leq 3 K^{2/9} \left( \mathbb{E}\sum_{k=1}^K \|G_k\|^2\right)^{1/3},
    \end{aligned}
\end{equation}
where the second inequality uses \eqref{eq:bound-bk}, the last inequality follows from $\frac{1}{a_{k+1}} = k^{2/3} \leq K^{2/3}$ for each $k\leq K-1$.  Plugging $c_1, c_2$ into \eqref{eq:bound-epslion} implies
\begin{equation}
\begin{aligned}
    \mathbb{E}\sum_{k=1}^{K} \|\epsilon_{k-1}\|^2 & \leq - \frac{\mathbb{E}[\|\epsilon_{T}\|^2]}{a_T}  + \frac{2}{3}\sum_{k=1}^K\mathbb{E}[\|\epsilon_k\|^2] + 3 K^{2/9} \left( \mathbb{E}\sum_{k=1}^K \|G_k\|^2\right)^{1/3} + 2\sigma^2 \sum_{k=1}^K a_k \\
    & \leq - \mathbb{E}[\|\epsilon_{T}\|^2]  + \frac{2}{3}\sum_{k=1}^K\mathbb{E}[\|\epsilon_k\|^2] + 3 K^{2/9} \left( \mathbb{E}\sum_{k=1}^K \|G_k\|^2\right)^{1/3} + 8 \sigma^2 K^{1/3},
    \end{aligned}
\end{equation}
where the second inequality follows from the fact that $\sum_{k=1}^K a_k = 1 + \sum_{k=1}^{K-1}k^{2/3} \leq 1 + 3K^{1/3} \leq 4 K^{1/3}$. Rearranging the above inequality yields \eqref{eq:bound:epsilon-k}. 
The proof is completed. 
\end{proof}

Building those lemmas, we now present the proof of Theorem \ref{Theorem of R2SRM}.

\begin{proof}[Proof of Theorem \ref{Theorem of R2SRM}]
It follows from the update rule of Algorithm \ref{alg:stomanial} and Lemma \ref{Euclidean vs manifold1} that
\begin{equation}\label{proof theorem1}
\begin{aligned}
& F_{{\mu_{k}}}(x_{k+1}) \leq F_{\mu_k}(x_k)-\tau_k\langle \grad F_{\mu_k}(x_k), G_k\rangle+\frac{\tau_k^{2}\mathcal{G}}{2\mu_k}\|G_k\|^{2} \\
&=F_{\mu_k}(x_k)-\frac{\tau_k}{2}\| \grad F_{\mu_k}(x_k)\|^2-\frac{\tau_k}{2}\|G_k\|^2+\frac{\tau_k}{2}\|G_k- \grad F_{\mu_k}(x_k)\|^2+\frac{\tau_k^{2}\mathcal{G}}{2\mu_k}\|G_k\|^2 \\
&\leq F_{\mu_k}(x_k)-\frac{\tau_k}{2}\| \grad F_{\mu_k}(x_k)\|^2+\frac{\tau_k}{2}\|\delta_k- \grad f(x_k)\|^2+\frac{\tau_k^{2}\mathcal{G}}{2\mu_k} \|G_k\|^2,
\end{aligned}
\end{equation}
where the first equality uses $2\langle a,b \rangle = \|a\|^2 + \|b\|^2 - \|a - b\|^2$ for any $a,b\in \mathbb{R}^n$. 
 Since $h$ is Lipschitz continuous with $\ell_h$ by Assumption \ref{assumption-h}, it follows from Lemma \ref{eq:moreau-bound-1} that for all  $ x\in\mathcal{M}$, we have
    \begin{equation*}
      	\begin{aligned}
			F_{\mu_{k+1}}(x)  & \leq F_{\mu_k}(x) + \frac{1}{2} (\mu_k - \mu_{k+1}) \frac{\mu_k}{\mu_{k+1}} \| \nabla h_{\mu_k}(\mathcal{A}x)  \|^2 \\
			& \leq F_{\mu_k}(x) + (\mu_k - \mu_{k+1}) \ell_h^2, 
		\end{aligned}
    \end{equation*}
    where the second equality utilizes \eqref{h bound} and $\frac{\mu_k}{\mu_{k+1}} = (\frac{k+1}{k})^{1/3} \leq 2$.
   Set $x = x_{k+1}$ in the above inequality and add it to \eqref{proof theorem1}, we  obtain
    \begin{equation}
        \label{eq:Fk+1-k-1-1}F_{\mu_{k+1}}(x_{k+1}) \leq F_{\mu_k}(x_k) - \frac{\tau_k}{2} \| \grad F_{\mu_k}(x_k) \|^2  +\frac{\tau_k}{2}\|\epsilon_k\|^{2}+\frac{\mathcal{G}\tau_k^2}{2 \mu_k} \|G_k\|^2   + (\mu_k - \mu_{k+1}) \ell_h^2.
    \end{equation}
    Denote $\Psi_k = F_{\mu_k}(x_k) - F_* + \mu_k \ell_h^2$, where $F_* = f(x_*) + h(c(x_*))$ is the optimal value and $x_*\in \Mcal$ is the optimal solution (refer to Assumption \ref{assumption-F}).  We show that exists a constant $B$ such that $\Psi_k \leq B$ for any $k\geq 1$.  It follows that
    \begin{equation}\label{def:B}
        \begin{aligned}
            \Psi_k =& f(x_k) - f(x_*)  + h(\prox_{\mu_k h}(c(x_k)))  - h(c(x_k))+h(c(x_k)) - h(c(x_*)) \\
            &+ \frac{1}{2\mu_k}\|c(x_k) - \prox_{\mu_k h}(c(x_k))\|^2  + \mu_k \ell_h^2 \\
           \leq  &  L_f \|x - x_*\| + \ell_h \| c(x_k) - \prox_{\mu_k h}(c(x_k)) \| + \|c(x_k) - c(x_*)\| \\
           &+ \frac{1}{2\mu_k}\|c(x_k) - \prox_{\mu_k h}(c(x_k))\|^2  + \mu_k \ell_h^2 \\
           \leq & (L_f + L_c)D + \frac{5}{2} \ell_h^2 \mu_k \leq (L_f + L_c)D + \frac{5}{2} \ell_h^2 \mu_1=:B, 
        \end{aligned}
    \end{equation}
    the first inequality uses the Lipschitz continuity of $f$ and $h$, the second inequality follows from \eqref{h bound}.     Using $\Psi_k$ and rearranging  \eqref{eq:Fk+1-k-1-1},  we get
    \begin{equation}
      \| \grad F_{\mu_k}(x_k) \|^2 \leq \| \epsilon_k \|^2 + \frac{2}{\tau_k} ( \Psi_k - \Psi_{k+1} ) + \frac{\mathcal{G} \tau_k}{ \mu_k}  \|G_k\|^2.
    \end{equation}
    Summing over $k$ gives
    \begin{equation}\label{eq:bound-gradF-k}
        \begin{aligned}
           & \sum_{k=1}^K \| \grad F_{\mu_k}(x_k)  \|^2 
             \leq  \sum_{k=1}^K \|\epsilon_k\|^2 -\frac{2}{\tau_K} \Psi_{K+1}+ 2\sum_{k=1}^K ( \frac{1}{\tau_k} - \frac{1}{\tau_{k-1}}) \Psi_k + \sum_{k=1}^K \frac{\mathcal{G}\tau_k}{\mu_k} \|G_k\|^2 \\
             \leq & \sum_{k=1}^K \|\epsilon_k\|^2 + 2B \sum_{k=1}^K ( \frac{1}{\tau_k} - \frac{1}{\tau_{k-1}}) + \mathcal{G}\sum_{k=1}^K \frac{a_{k+1}^{1/3}\|G_k\|^2}{\mu_k(\sum_{i=1}^k \|G_i\|^2)^{1/3}} \\
             \leq & \sum_{k=1}^K \|\epsilon_k\|^2 + 2B \frac{1}{\tau_K} + 
\mathcal{G} K^{1/9} \sum_{k=1}^K \frac{\|G_k\|^2}{(\sum_{i=1}^k \|G_i\|^2)^{1/3}} \\
             \leq & \sum_{k=1}^K \|\epsilon_k\|^2 + 2B a_{K+1}^{-1/3}(\sum_{k=1}^K\|G_k\|^2)^{1/3} + 
\mathcal{G} K^{1/9} \sum_{k=1}^K \frac{\|G_k\|^2}{(\sum_{i=1}^k \|G_i\|^2)^{1/3}} \\
             \leq & \sum_{k=1}^K \|\epsilon_k\|^2 + 2B K^{2/9}(\sum_{i=1}^K\|G_k\|^2)^{1/3} + 
 \frac{3}{2}\mathcal{G} K^{1/9} (\sum_{k=1}^K \|G_k\|^2  )^{2/3},
        \end{aligned}
    \end{equation}
where the second inequality uses the update rule of $\tau_k$ in \eqref{eq:constant-cnd} and \eqref{def:B}, the third inequality utilizes $\frac{a_{k+1}^{1/3}}{\mu_k} = k^{1/9} \leq K^{1/9}$ for any $k\leq K$,  the last inequality follows from Lemma \ref{lem:bound-bk}.   Taking expectation on both sides of \eqref{eq:bound-gradF-k} yields
\begin{equation}\label{eq:bound-gradF-k-1}
   \mathbb{E} \sum_{k=1}^K \| \grad F_{\mu_k}(x_k)  \|^2  \leq \mathbb{E}\sum_{k=1}^K \|\epsilon_k\|^2 + 2B K^{2/9} \mathbb{E}(\sum_{k=1}^K\|G_k\|^2)^{1/3} + 
 \frac{3}{2}\mathcal{G} K^{1/9} \mathbb{E}(\sum_{k=1}^K \|G_k\|^2  )^{2/3}.
\end{equation}

Now we bound the term $\mathbb{E}\sum_{k=1}^K \|\grad F_{\mu_k}(x_k)\|^2$ by considering two cases.  

\textbf{Case 1: Assume that $\mathbb{E}\sum_{k=1}^K \|\epsilon_k\|^
2 \geq \frac{1}{2}\mathbb{E}\sum_{k=1}^K \|\grad F_{\mu_k}(x_k)\|^2$}. It follows from the update rule of $G_k$ that
\begin{equation}\label{eq:G_k-formula}
\begin{aligned}
  G_k & =   \delta_k - \grad f(x_k) + \grad f(x_k) + \mathcal{P}_{T_{x_k} \Mcal}( \nabla c(x_k)^\top \nabla h_{\mu_k}(c(x_k)) ) \\
& = \epsilon_k + \grad F_{\mu_k}(x_k).
  \end{aligned}
\end{equation}
Combining the condition of Case 1 with $\|G_k\|^2 \leq 2\|\grad F_{\mu_k}(x_k)\|^2 + 2\|\epsilon_k\|^2$, implies that $\mathbb{E}\sum_{k=1}^K \| G_k \|^2 \leq 6\mathbb{E}\sum_{k=1}^K \|\epsilon_k\|^2$. Plugging this inside \eqref{eq:bound:epsilon-k} yields
\begin{equation}
    \mathbb{E}\sum_{k=1}^K \|\epsilon_k\|^2 \leq 18\tilde{L}^2 K^{2/9} \left( 6\mathbb{E}\sum_{k=1}^K \|\epsilon_k\|^2\right)^{1/3} + 24 \sigma^2 K^{1/3}.
\end{equation}
Letting $\mathbb{E}\sum_{k=1}^K \|\epsilon_k\|^2$ be the term $x$ in Lemma \ref{lem:bound-x}, and using \eqref{eq:bound-x},  the above immediately implies that 
\begin{equation}
\begin{aligned}
    \mathbb{E}\sum_{k=1}^K \|\epsilon_k\|^2 &\leq2 (\frac{4}{3})^{1/2} ( 18 \cdot 6^{1/3}\tilde{L}^2K^{2/9})^{3/2} + 48 \sigma^2 K^{1/3}\\
    & \leq (432\sqrt{3} \tilde{L}^3 + 48\sigma^2)K^{1/3}.
    \end{aligned}
\end{equation}
Combining with the condition of Case 1 yields $ \mathbb{E}\sum_{k=1}^K \|\grad F_{\mu_k}(x_k)\|^2 \leq (864 \sqrt{3} \tilde{L}^3 + 96\sigma^2)K^{1/3}$.

\textbf{Case 2: Assume that $\mathbb{E}\sum_{k=1}^K \|\epsilon_k\|^2 \leq \frac{1}{2}\mathbb{E}\sum_{k=1}^K \|\grad F_{\mu_k}(x_k)\|^2$.} Combining the condition of Case 2 with $\|G_k\|^2 \leq 2\|\grad F_{\mu_k}(x_k)\|^2 + 2\|\epsilon_k\|^2$, implies that $\mathbb{E}\sum_{k=1}^K \| G_k \|^2 \leq 3\mathbb{E}\sum_{k=1}^K \|\grad F_{\mu_k}(x_k)\|^2$. Plugging this inside \eqref{eq:bound-gradF-k-1} yields
\begin{equation}
\begin{aligned}
    \mathbb{E}\sum_{k=1}^K \|\grad F_{\mu_k}(x_k)\|^2 
\leq &  \frac{1}{2} \mathbb{E}\sum_{k=1}^K \|\grad F_{\mu_k}(x_k)\|^2 + 2B K^{2/9}(3\mathbb{E}\sum_{i=1}^K\|\grad F_{\mu_k}(x_k)\|^2)^{1/3}  \\
&+ 
 \frac{3}{2}\mathcal{G} K^{1/9} (3\mathbb{E}\sum_{i=1}^K\|\grad F_{\mu_k}(x_k)\|^2  )^{2/3},
 \end{aligned}
\end{equation}
where we also used Jensen's inequality with respect to the concave functions $z^{1/3}$ and $z^{2/3}$ defined over $\mathbb{R}_+$. This implies that
\begin{equation}
\begin{aligned}
&    \mathbb{E}\sum_{k=1}^K \|\grad F_{\mu_k}(x_k)\|^2 \\ \leq & 8BK^{2/9} ( \mathbb{E}\sum_{k=1}^K \|\grad F_{\mu_k}(x_k)\|^2)^{1/3} + 9 \mathcal{G}K^{1/9}(\mathbb{E}\sum_{k=1}^K \|\grad F_{\mu_k}(x_k)\|^2 )^{2/3}.
\end{aligned}
\end{equation}
Using Lemma \ref{lem:bound-x}, the above immediately implies $$
\begin{aligned}
\mathbb{E}\sum_{k=1}^K \|\grad F_{\mu_k}(x_k)\|^2&  \leq  2 (\frac{4}{3})^{1/2}
 (8B K^{2/9})^{3/2} + 2(\frac{8}{3})^2 (9\mathcal{G} K^{1/9})^{3}    \\
 & \leq (64 B^{3/2} + 9^4 \mathcal{G}^3) K^{1/3}.
 \end{aligned}$$ 
Combining with two cases completes the proof.  
\end{proof}

\subsection{Proof of Section \ref{sec:indicator}}

        The following lemma establishes a relationship between $g(x_{k+1})$ and $g(x_k)$, which will be used to
derive bounds for $\text{dist}^2(c(x),\mathcal{C})$. 
     \begin{lemma}
        Suppose that Assumptions \ref{general assumption}, \ref{assum:h-2} and \ref{assum:error-bound} hold, and $\{x_k\}$ is generated by Algorithm \ref{alg:stomanial-1}. If $c_{\tau} \leq \frac{1}{L_g}$, where $L_g$ is defined in \eqref{def:L}, then we have
        \begin{equation}\label{eq:diff-gk-k-1}
            g(x_{k+1}) - g(x_k) \leq  -   \frac{\tau_k \zeta^2}{2\mu_k} (g(x_k))^{\theta}  +  \frac{L_f^2}{2} \tau_k\mu_k.
        \end{equation}
    \end{lemma}

\begin{proof}
Before proving main result, we need to show that $g$ is retr-smooth over $\mathcal{M}$. Let us denote $r(x): = c(x) - \mathcal{P}_{\mathcal{C}}(c(x))$.   We first show that $g(x)$ is smooth over the convex hull $\text{conv}(\Mcal)$. 
Note that  Since $\mathcal{M}$ is a compact submanifold, we know from $r$ is continuous that there exists $C_r>0$ such that $\|r(x)\| \leq C_r$ for $x\in \mathcal{M}$. 
It follows from Lemma \ref{assumption-B} that for any $x,y\in \mathcal{M}$, it holds 
\begin{equation}
\begin{aligned}
& \|\nabla g(x) - \nabla g(y) \| = \| \nabla c(x)^\top r(x) - \nabla c(y)^\top r(y) \| \\
    \leq & \|\nabla c(x)^\top r(x) - \nabla c(x)^\top r(y)  \| + \|\nabla c(x)^\top r(y) - \nabla c(y)^\top r(y)  \| \\
    \leq &  L_c \|r(x) - r(y)\| + C_r \| \nabla c(x) - \nabla c(y)\| \\
     \leq & (L_c + C_r L_{\nabla c}) \|x-y\|.
    \end{aligned}
\end{equation}
Therefore, $g$ is $L_c + C_r L_{\nabla c}$-smooth on $\mathcal{M}$. Moreover, the gradient of $g$ is bounded: $\|\nabla g(x)\| \leq \|\nabla c(x)\| \|r(x)\| = L_c C_r$.  For notational convenience, we define
\begin{equation}\label{def:L}
    L_g:=\alpha^2 (L_c + C_r L_{\nabla c}) + 2 ( L_c C_r) \beta.
\end{equation}
 It follows from Proposition \ref{prop:retr} that $g$ is retr-smooth with constant $L_g$. 
By Algorithm \ref{alg:stomanial-1}, we have that
\begin{equation}
        \begin{aligned}
          &  g(x_{k+1})   \leq g(x_k) - \frac{\tau_k}{\mu_k} \langle \grad g(x_k), \mu_k  G_k \rangle + \frac{\tau_k^2 L_g}{2}\|G_k\|^2 \\
            & \leq g(x_k) -  \frac{\tau_k}{2\mu_k} \left(  
\|\grad g(x_k) \|^2 + \mu_k^2 \|G_k\|^2 - \| \grad g(x_k) - \mu_k G_k \|^2 \right)+ \frac{\tau_k^2 L_g}{2}\|G_k\|^2 \\
& \leq  g(x_k) -   \frac{\tau_k \zeta^2}{2\mu_k} \mathrm{dist}^{2\theta}(c(x_k),\mathcal{C}) -   ( \frac{\tau_k\mu_k}{2} - \frac{\tau_k^2 L_g}{2} ) \|G_k\|^2 +  \frac{\tau_k}{2\mu_k} \| \grad g(x_k) - \mu_k G_k \|^2 \\
& \leq  g(x_k) -   \frac{\tau_k \zeta^2}{2\mu_k} (g(x_k))^{\theta}  +  \frac{L_f^2}{2} \tau_k\mu_k,
        \end{aligned}
    \end{equation}
    where the second inequality uses $2\langle a,b \rangle = \|a\|^2 + \|b\|^2 - \|a - b\|^2$ for any $a,b\in \mathbb{R}^n$. the third inequality follows from  Assumption \ref{assum:error-bound}, the last inequality utilizes $\| \grad g(x_k) - \mu_k G_k\|^2 = \|\mu_k \mathcal{P}_{T_{x_k}\mathcal{M}}(\delta_k)   \|^2 \leq \mu_k^2 \|\delta_k\|^2 \leq \mu_k^2 L_f^2$ by \eqref{def:delta:indicator} and $\tau_k \leq \frac{\mu_k}{L_g}$ by  $c_{\tau} \leq \frac{1}{L_g}$.

\end{proof}

\begin{lemma}
     Suppose that Assumptions \ref{general assumption}, \ref{assum:h-2} and \ref{assum:error-bound} hold, and $\{x_k\}$ is generated by Algorithm \ref{alg:stomanial-1}.  Let us denote $\tilde{K} := 
 \lceil \frac{8\omega}{c_{\tau} \zeta^2}  \rceil$ and
    \begin{equation}\label{eq:cond-mu}
    C_{g} := \left\{
    \begin{array}{cc}
     \max\{  \frac{1}{2}L_c D \tilde{K}^{\frac{2\omega}{\theta}}, \frac{2L_f^2}{\zeta^2}\}  &  \text{if} ~~~\theta = 1 \\
        \max \left\{1,\frac{1}{2}L_c D \tilde{K}^{\frac{2\omega}{\theta}}, (\frac{8\omega}{\theta c_{\tau}\zeta^2})^{\frac{1}{\theta-1}} \right\} & \text{otherwise} 
    \end{array}
    \right.
\end{equation}
If $c_{\tau} \leq \frac{1}{L_g}$, where $L_g$ is defined in \eqref{def:L}, and  $\omega \leq \frac{1}{2}$, then we have that for any $k\geq \tilde{K}$,
    \begin{equation}\label{eq:bound:g}
        \text{dist}^2(c(x_k), \mathcal{C}) \leq 2C_g k^{-\frac{2\omega}{\theta}}.
    \end{equation}
\end{lemma}

\begin{proof}
We now prove this by induction. For $k = \tilde{K}$, it holds that
\begin{equation}
    g(x_{\tilde{K}})  = \frac{1}{2}\text{dist}^2(c(x_{\tilde{K}}), \mathcal{C}) \leq \frac{1}{2} \|c(x_{\tilde{K}}) - c(x_*)\|^2 \leq \frac{1}{2}  L_c D  \leq  C_g \tilde{K}^{-\frac{2\omega}{\theta}}, 
\end{equation}
where the last inequality uses $C_g \geq \frac{1}{2}L_c D \tilde{K}^{\frac{2\omega}{\theta}}$. 
Now suppose for induction that $g(x_k) \leq C_g  k^{-\frac{2\omega}{\theta}}$ holds for some $k\geq \tilde{K}$. Then
\begin{equation}\label{temp:g}
    \begin{aligned}
        g(x_{k+1}) - C_g (k+1)^{-\frac{2\omega}{\theta}} & \leq g(x_k)  -   \frac{\tau_k \zeta^2}{2\mu_k} (g(x_k))^{\theta}  +  \frac{L_f^2}{2} \tau_k\mu_k- C_g (k+1)^{-\frac{2\omega}{\theta}} \\
        & \leq C_g  k^{-\frac{2\omega}{\theta}} -\frac{1}{2} c_{\tau}\zeta^2 C_g^{\theta} k^{-2\omega} + \frac{c_\tau  L_f^2}{2} k^{-2\omega} - C_g (k+1)^{-\frac{2\omega}{\theta}} \\
        & \leq C_g \frac{2\omega}{\theta} k^{\frac{-2\omega}{\theta}-1} - \frac{1}{2}c_{\tau}\zeta^2 C_g^{\theta} k^{-2\omega} + \frac{c_\tau   L_f^2}{2} k^{-2\omega},
    \end{aligned}
\end{equation}
where the first inequality uses \eqref{eq:diff-gk-k-1}, the second inequality is due to that $\mu_k =  k^{-\omega}$ and $\tau_k = c_{\tau} k^{-\omega}$, the third inequality follows from $(k+1)^{-\frac{2\omega}{\theta}} - k^{-\frac{2\omega}{\theta}} \geq \frac{-2\omega}{\theta} k^{-\frac{2\omega}{\theta}-1}$ thanks to the convexity of $t^{-\frac{2\omega}{\theta}}$.

Now we analyze \eqref{temp:g} by considering two cases for $\theta$. When $\theta = 1$, it holds that 
\begin{equation}
    \begin{aligned}
       & g(x_{k+1}) - C_g (k+1)^{-\frac{2\omega}{\theta}} \\
         \leq & 2\omega C_g  k^{-2\omega-1} - \frac{1}{2}c_{\tau}\zeta^2 C_g k^{-2\omega} + \frac{c_\tau   L_f^2}{2} k^{-2\omega} \\
         \leq & 2\omega C_g  k^{-2\omega-1} - \frac{1}{4}c_{\tau}\zeta^2 C_g k^{-2\omega}  =  (2\omega C_g k^{-1} - \frac{1}{4}c_{\tau}\zeta^2 C_g) k^{-2\omega} \leq 0,
    \end{aligned}
\end{equation}
where the second inequality uses $C_g \geq \frac{ 2L_f^2}{\zeta^2}$, the last inequality utilizes $k \geq 
 \lceil \frac{8\omega}{c_{\tau} \zeta^2}  \rceil $. For $\theta >1$, since that $\omega \leq  \frac{1}{2} \leq \frac{\theta}{2(\theta-1)}$, one obtains $\frac{2\omega}{\theta} + 1 \geq   2\omega$. Therefore, it follows from \eqref{temp:g} that
 \begin{equation}
     \begin{aligned}
         g(x_{k+1}) - C_g (k+1)^{-\frac{2\omega}{\theta}} 
         \leq & C_g \frac{2\omega}{\theta} k^{-2\omega} -\frac{1}{2} c_{\tau}\zeta^2 C_g^{\theta} k^{-2\omega} + \frac{c_\tau   L_f^2}{4} k^{-2\omega} \\
          \leq &  -\frac{1}{2} c_{\tau}\zeta^2 C_g^{\theta} k^{-2\omega} + \frac{c_\tau   L_f^2}{2} k^{-2\omega} \leq 0,
     \end{aligned}
 \end{equation}
where the second inequality utilizes $C_g \geq (\frac{8\omega}{\theta c_{\tau}\zeta^2})^{\frac{1}{\theta-1}}$, the last inequality uses $C_g \geq \max\{1, \frac{ 2L_f^2}{\zeta^2}\}$. In conclusion, for any $k \geq 
\tilde{K}$, it holds that
\begin{equation}
    g(x_k) \leq C_g k^{\frac{-2\omega}{\theta}}. 
\end{equation}
According to the definition of $g$ in \eqref{def:g}, we obtain \eqref{eq:bound:g} and  complete the proof.

\end{proof}

Similar to Lemma \ref{Estimation error bound}, we now bound the estimation error of the recursive momentum estimator in Algorithm \ref{alg:stomanial-1}.  
\begin{lemma}[Estimation error bound]\label{Estimation error bound-2}
Suppose Assumptions \ref{general assumption} holds and consider Algorithm \ref{alg:stomanial-1}. Let us denote $\epsilon_k = \delta_k - \grad f(x_k)$. Then the expected estimation error of the estimator is bounded by
\begin{equation}\label{eq:recursive:epsilon-1}
    \mathbb{E}[ \|\epsilon_k\|^2] \leq   (1-a_k)^2\big(1+4\tau_{k-1}^2\tilde{L}^2\big)\|\epsilon_{k-1}\|^2+4(1-a_k)^2\tau_{k-1}^2\tilde{L}^2\| \grad F_{\mu_{k-1}}(x_{k-1})\|^2+2a_k^2\sigma^2.
\end{equation}
\end{lemma}

\begin{proof}
    The only difference between the two algorithms lies in the stochastic gradient estimator. Note that $\|\grad f(x_k)\| \leq \|\nabla f(x_k)\| \leq L_f $ by Assumption \ref{general assumption} and hence $\grad f(x_k) = \mathcal{P}_{\mathcal{B}_{L_f}}(\grad f(x_k))$. Using the the nonexpansiveness of $\mathcal{P}_{\mathcal{B}_{L_f}}$, we obtain 
    \begin{equation}
    \begin{aligned}
        &  \|\delta_k - \grad f(x_k)\|  \\
        = & \|\mathcal{P}_{\mathcal{B}_{L_f}}( \grad \tilde{f}(x_k,\xi_k) + (1-a_{k-1})\mcT_{x_{k-1}}^{x_k}(\delta_{k-1} - \grad \tilde{f}(x_{k-1},\xi_k)) ) - \mathcal{P}_{\mathcal{B}_{L_f}}(\grad f(x_k))\| \\
        \leq  & \|\grad \tilde{f}(x_k,\xi_k) + (1-a_{k-1})\mcT_{x_{k-1}}^{x_k}(\delta_{k-1} - \grad \tilde{f}(x_{k-1},\xi_k)) ) - \grad f(x_k)\|.
        \end{aligned}
    \end{equation}
    Therefore, we can obtain the same result as in Lemma \ref{Estimation error bound}, namely,
    $$
    \mathbb{E}[ \|\epsilon_k\|^2] \leq   (1-a_k)\mathbb{E}[\|\epsilon_{k-1}\|^2]+2a_k^2\sigma^2+2\tilde{L}^2\tau_{k-1}^2\mathbb{E}[\|G_{k-1}\|^2].
    $$
    Moreover, since $G_{k-1} = \epsilon_{k-1} + \grad F_{\mu_{k-1}}(x_{k-1})$, substituting this into the above inequality completes the proof.
\end{proof}

Finally, we give the proof of Theorem \ref{Theorem of R2SRM-1}.
\begin{proof}[Proof of Theorem \ref{Theorem of R2SRM-1}]
Similar to the proof in Theorem \ref{Theorem of R2SRM}, one has from Lemma \ref{Euclidean vs manifold1} that
\begin{equation}\label{proof theorem1-1}
\begin{aligned}
& F_{{\mu_{k}}}(x_{k+1}) \leq F_{\mu_k}(x_k)-\langle \grad F_{\mu_k}(x_k), \tau_kG_k\rangle+\frac{\tau_k^{2}\mathcal{G}}{2\mu_k}\|G_k\|^{2} \\
&=F_{\mu_k}(x_k)-\frac{\tau_k}{2}\| \grad F_{\mu_k}(x_k)\|^2-\frac{\tau_k}{2}\|G_k\|^2+\frac{\tau_k}{2}\|G_k- \grad F_{\mu_k}(x_k)\|^2+\frac{\tau_k^{2}\mathcal{G}}{2\mu_k}\|G_k\|^2 \\
&\leq F_{\mu_k}(x_k)-\frac{\tau_k}{2}\| \grad F_{\mu_k}(x_k)\|^2+\frac{\tau_k}{2}\|\delta_k- \grad f(x_k)\|^2, 
\end{aligned}
\end{equation}
where the last inequality uses $\tau_k \leq \frac{\mu_k}{\mathcal{G}}$.  Since $\tau_k = c_{\tau}(k+1)^{-\omega}$ and $\mu_k =  k^{-\omega}$, we can ensure this condition is met by requiring $c_{\tau} \leq 1/\mathcal{G}$. 
 It follows from  \eqref{eq:moreau-bound-2} that for all  $ x\in\mathcal{M}$
    \begin{equation*}
      	\begin{aligned}
			F_{\mu_{k+1}}(x)  & \leq F_{\mu_k}(x) + \frac{1}{2} (\frac{1}{\mu_{k+1}} - \frac{1}{\mu_k}) \text{dist}^2(c(x),\mathcal{C})\\
            & \leq F_{\mu_k}(x) + \frac{\omega}{2}k^{\omega-1}  \text{dist}^2(c(x),\mathcal{C}),
		\end{aligned}
    \end{equation*}
    where the second equality utilizes $\mu_k = k^{-\omega}$ and $(k+1)^{\omega} - k^{\omega} \leq \omega k^{\omega-1}$ by the concavity of $t^{\omega}$, where $\omega\leq \frac{1}{2}$.  
   Set $x = x_{k+1}$ in the above inequality and add it to \eqref{proof theorem1}, we  obtain
    \begin{equation}
        \label{eq:Fk+1-k-1}F_{\mu_{k+1}}(x_{k+1}) \leq F_{\mu_k}(x_k) - \frac{\tau_k}{2} \| \grad F_{\mu_k}(x_k) \|^2  +\frac{\tau_k}{2}\|\delta_k- \grad f(x_k)\|^{2} + \frac{\omega}{2}k^{\omega-1} \text{dist}^2(c(x_{k+1}),\mathcal{C}).
    \end{equation}
We construct a Lyapunov function
\begin{equation*}
\Phi_k:=\mbE[F_{\mu_k}(x_k)]+\frac {C}{\tau_{k-1}}\mbE\|\delta_k- \grad f(x_k)\|^2,
\end{equation*} where $C = \frac{1}{16\tilde{L}^2}$. Using \eqref{eq:recursive:epsilon-1} and combining with  \eqref{eq:Fk+1-k-1} yields
\begin{equation}\label{eq:sk+1-sk-1-1}
\begin{aligned}
\Phi_{k+1}   \leq & F_{\mu_k}(x_k) - \frac{\tau_k}{2} \| \grad F_{\mu_k}(x_k) \|^2  +\frac{\tau_k}{2}\|\epsilon_k\|^{2} + \frac{\omega}{2}k^{\omega-1} \text{dist}^2(c(x_{k+1}),\mathcal{C}) \\
 & +  \frac{C(1-a_{k+1})^2\big(1+4\tau_{k}^2\tilde{L}^2\big)}{\tau_k} \mbE \|\epsilon_k\|^2\\
&+\frac{4C(1-a_{k+1})^2\tau_{k}^2\tilde{L}^2}{\tau_k}\| \grad F_{\mu_{k}}(x_{k})\|^2+\frac{2Ca_{k+1}^2\sigma^2}{\tau_k} \\
\leq & \Phi_k + \underbrace{ \left(\frac{\tau_k}{2}   +  \frac{C(1-a_{k+1})^2\big(1+4\tau_{k}^2\tilde{L}^2\big)}{\tau_k}  -   \frac{C}{\tau_{k-1}}\right)}_{c_1} \|\epsilon_k\|^{2} \\
& -  \underbrace{\left( \frac{\tau_k}{2} -  4C(1-a_{k+1})^2\tau_{k}\tilde{L}^2\right)}_{c_2}\| \grad F_{\mu_{k}}(x_{k})\|^2 \\
& +\frac{\omega}{2}k^{\omega-1} \text{dist}^2(c(x_{k+1}),\mathcal{C}) + \frac{2Ca_{k+1}^2\sigma^2}{\tau_k}.
\end{aligned}
\end{equation}
Now we bound $c_1$ and $c_2$ respectively. Firstly, 
\begin{equation}
    \begin{aligned}
        c_1 & \leq \frac{\tau_k}{2} + C  (  \frac{1}{\tau_k} - \frac{1}{\tau_{k-1}} + +4\tau_k\tilde{L}^2-\frac{a_{k+1}}{\tau_k}  ) \\
        & \leq \frac{\tau_k}{2} + C(\frac{\omega k^{2\omega-1}}{c_{\tau}^2}  \tau_k+4\tau_k\tilde{L}^2-\frac{c_a}{c_\tau^2} \tau_k ) \\
        & \leq ( \frac{1}{2} +\frac{C\omega k^{2\omega-1}}{c_{\tau}^2} + 4C\tilde{L}^2 - \frac{c_a}{c_\tau^2}) \tau_k   \leq 0
    \end{aligned}
\end{equation}
where the first inequality uses the fact that $(1-a_{k+1})^2\leq 1-a_{k+1}\leq 1$, the second inequality utilizes $\frac{1}{\tau_k}-\frac{1}{\tau_{k-1}}\leq\frac{\omega}{c_{\tau}}k^{\omega-1} = \frac{\omega}{c_{\tau}^2} k^{2\omega-1} \tau_k $, the last inequality is due to $\omega \leq \frac{1}{2}$ and 
$$
c_a = c_{\tau}^2 ( \frac{1}{2} +  \frac{C}{2c_{\tau}^2} + 4C\tilde{L}^2).
$$
For $c_2$, since $C = \frac{1}{16\tilde{L}^2}$, one has that
\begin{equation}
    c_2 \geq \frac{\tau_k}{2} - 4C\tilde{L}^2 \tau_k  = \frac{\tau_k}{4}.
\end{equation}
Combining with $c_1,c_2$, and plugging \eqref{eq:bound:g} into \eqref{eq:sk+1-sk-1-1} yields
\begin{equation*}
\begin{aligned}
\Phi_{k+1}-\Phi_k 
\leq & -\frac{\tau_k}4\mbE\| \grad F_{\mu_k}(x_k)\|^{2}+ C_g\omega    k^{\omega-1-\frac{2\omega}{\theta}} + \frac{\sigma^2 c_a^2}{8\tilde{L}^2 c_{\tau}} k^{-3\omega}.
\end{aligned}
\end{equation*} 
Telescoping this result from $k=\tilde{K}, \ldots,K$ gives
\begin{equation}\label{eq:divide-rk-r1-1}
\begin{aligned}
\Phi_{K+1}-\Phi_{\tilde{K}}\leq& -\frac{1}{4}\sum_{k=\tilde{K}}^{K}\tau_k\mbE\| \grad F(x_k)\|^2+C_g\omega \sum_{k=\tilde{K}}^K   k^{\omega-1-\frac{2\omega}{\theta}} + \frac{\sigma^2 c_a^2}{8\tilde{L}^2 c_{\tau}} \sum_{k=\tilde{K}}^K k^{-3\omega}. 
\end{aligned}
\end{equation}
Rearranging the above inequality yields 
\begin{equation}\label{eq:gradF-dist-indicator-1}
     \begin{aligned}
       & \sum_{k=\tilde{K}}^K \tau_k  \mathbb{E}\| \grad F_{\mu_k}(x_k)\|^{2}  \\
      \leq & 4(\Phi_{\tilde{K}} - \Phi_{K+1}) +4C_g\omega \sum_{k=\tilde{K}}^K   k^{\omega-1-\frac{2\omega}{\theta}} + \frac{4\sigma^2 c_a^2}{8\tilde{L}^2 c_{\tau}} \sum_{k=\tilde{K}}^K k^{-3\omega} \\
      \leq & 4(\Phi_{\tilde{K}} - \Phi_{K+1}) +4C_g\omega  \max(K^{\omega-\frac{2\omega}{\theta}},1) \sum_{k=\tilde{K}}^K   k^{-1} + \frac{4\sigma^2 c_a^2}{8\tilde{L}^2 c_{\tau}} \sum_{k=\tilde{K}}^K k^{-1} \\
      \leq & 4(\Phi_{\tilde{K}} - \Phi_{K+1}) + (4C_g\omega + \frac{4\sigma^2 c_a^2}{8\tilde{L}^2 c_{\tau}}) \max\{  K^{\omega -\frac{2\omega}{\theta}},  1\}   \ln(K)
     \end{aligned}
 \end{equation}
where we uses $\omega  \geq \frac{1}{3}$ by $\theta \geq 1$.
 % Combining with \eqref{eq:divide-rk-r1-1} yields 
 % \begin{equation}\label{eq:gradF-dist-indicator}
 %     \begin{aligned}
 %       & \sum_{k=1}^K \tau_k \left( \mathbb{E}\| \grad F_{\mu_k}(x_k)\|^{2} + \text{dist}^2(c(x_{k}), \mathcal{C}) \right) \\
 %      \leq & 4(\Phi_1 - \Phi_{K+1}) +4C_g\omega \sum_{k=1}^K   k^{\omega-1-\frac{2\omega}{\theta}} + \frac{4\sigma^2 c_a^2}{8\tilde{L}^2 c_{\tau}} \sum_{k=1}^K k^{-3\omega} +2C_g c_{\tau} \sum_{k=1}^K  k^{-\frac{2\omega}{\theta} - \omega}\\
 %      \leq & 4(\Phi_1 - \Phi_{K+1}) +4C_g\omega  \max(K^{\omega-\frac{2\omega}{\theta}},1) \sum_{k=1}^K   k^{-1} + \frac{4\sigma^2 c_a^2}{8\tilde{L}^2 c_{\tau}} \sum_{k=1}^K k^{-1} \\
 %      &+2C_g c_{\tau} \max(K^{-\frac{2\omega}{\theta} - \omega +1},1) \sum_{k=1}^K  k^{-1}\\
 %      \leq & 4(\Phi_1 - \Phi_{K+1}) + (4C_g\omega + 2C_g c_{\tau} +\frac{4\sigma^2 c_a^2}{8\tilde{L}^2 c_{\tau}}) \max\{  K^{\omega -\frac{2\omega}{\theta}}, K^{-\frac{2\omega}{\theta} - \omega +1}, 1\}   \ln(K)
 %     \end{aligned}
 % \end{equation}
%where we uses $\omega  \geq \frac{1}{3}$ by $\theta \geq 1$. 
Finally,  we give a lower bound of $\Phi_k$:
       \begin{equation}
       \begin{aligned}
          \Phi_k  & = f(x_k) + \frac{1}{2\mu_k}\mathrm{dist}^2(c(x_k),\mathcal{C}) +\frac {1}{16\tilde{L}^2\tau_{k-1}}\mbE\|\delta_k- \grad f(x_k)\|^2   \geq F_* =:\Phi_*. 
       \end{aligned}
       \end{equation}
    Combining with \eqref{eq:gradF-dist-indicator-1} completes the proof. 

        \end{proof}

\section{Conclusion}

In this work, we proposed two Riemannian stochastic smoothing algorithms for solving nonsmooth optimization problems on manifolds, tailored respectively to the cases where the nonsmooth term \( h \) is Lipschitz continuous or an indicator function of a convex set. Both algorithms leverage dynamic smoothing and online momentum-based variance reduction, and feature a single-loop structure with low per-iteration cost and constant sample size per iteration. For the Lipschitz case, our algorithm achieves the optimal iteration complexity of \( \mathcal{O}(\epsilon^{-3}) \), improving upon existing stochastic methods in the Riemannian setting. In the constrained case, under a mild error bound condition, we establish a near-optimal complexity of \( \tilde{\mathcal{O}}(\epsilon^{-\max\{\theta+2, 2\theta\}}) \). In future work, we plan to conduct numerical experiments to evaluate the performance of the proposed algorithms.

\section*{Acknowledgments}  This research is supported by the National Natural Science Foundation of China (NSFC) grants 12401419, 12071398, the Key Program of National Natural Science of China 12331011, the Major Research Plan of National Natural Science Foundation of China 92473208, the Innovative Research Group Project of Natural Science Foundation of Hunan Province of China 2024JJ1008.
\bibliographystyle{siamplain}
\bibliography{ref}

\begin{thebibliography}{10}

\bibitem{absil2017collection}
{\sc P.-A. Absil and S.~Hosseini}, {\em A collection of nonsmooth {R}iemannian
  optimization problems}, Nonsmooth Optimization and Its Applications,  (2019),
  pp.~1--15.

\bibitem{AbsMahSep2008}
{\sc P.-A. Absil, R.~Mahony, and R.~Sepulchre}, {\em Optimization Algorithms on
  Matrix Manifolds}, Princeton University Press, Princeton, NJ, 2008.

\bibitem{alacaoglu2024complexity}
{\sc A.~Alacaoglu and S.~J. Wright}, {\em Complexity of single loop algorithms
  for nonlinear programming with stochastic objective and constraints}, in
  International Conference on Artificial Intelligence and Statistics, PMLR,
  2024, pp.~4627--4635.

\bibitem{beck2023}
{\sc A.~Beck and I.~Rosset}, {\em A dynamic smoothing technique for a class of
  nonsmooth optimization problems on manifolds}, SIAM Journal on Optimization,
  33 (2023), pp.~1473--1493, \url{https://doi.org/10.1137/22M1489447}.

\bibitem{berahas2021sequential}
{\sc A.~S. Berahas, F.~E. Curtis, D.~Robinson, and B.~Zhou}, {\em Sequential
  quadratic optimization for nonlinear equality constrained stochastic
  optimization}, SIAM Journal on Optimization, 31 (2021), pp.~1352--1379.

\bibitem{bohm2021variable}
{\sc A.~B{\"o}hm and S.~J. Wright}, {\em Variable smoothing for weakly convex
  composite functions}, Journal of Optimization Theory and Applications, 188
  (2021), pp.~628--649.

\bibitem{bonnabel2013stochastic}
{\sc S.~Bonnabel}, {\em Stochastic gradient descent on {R}iemannian manifolds},
  IEEE Transactions on Automatic Control, 58 (2013), pp.~2217--2229.

\bibitem{grocf}
{\sc N.~Boumal, P.-A. Absil, and C.~Cartis}, {\em Global rates of convergence
  for nonconvex optimization on manifolds}, IMA Journal of Numerical Analysis,
  39 (2019), pp.~1--33.

\bibitem{cambier2016robust}
{\sc L.~Cambier and P.-A. Absil}, {\em Robust low-rank matrix completion by
  {R}iemannian optimization}, SIAM Journal on Scientific Computing, 38 (2016),
  pp.~S440--S460.

\bibitem{curtis2024worst}
{\sc F.~E. Curtis, M.~J. O’Neill, and D.~P. Robinson}, {\em Worst-case
  complexity of an {SQP} method for nonlinear equality constrained stochastic
  optimization}, Mathematical Programming, 205 (2024), pp.~431--483.

\bibitem{curtis2021inexact}
{\sc F.~E. Curtis, D.~P. Robinson, and B.~Zhou}, {\em Inexact sequential
  quadratic optimization for minimizing a stochastic objective function subject
  to deterministic nonlinear equality constraints}, arXiv preprint
  arXiv:2107.03512,  (2021).

\bibitem{cutkosky2019momentum}
{\sc A.~Cutkosky and F.~Orabona}, {\em Momentum-based variance reduction in
  non-convex sgd}, Advances in neural information processing systems, 32
  (2019).

\bibitem{demidovich2024streamlining}
{\sc Y.~Demidovich, G.~Malinovsky, and P.~Richt{\'a}rik}, {\em Streamlining in
  the {Riemannian} realm: Efficient {Riemannian} optimization with loopless
  variance reduction}, arXiv preprint arXiv:2403.06677,  (2024).

\bibitem{deng2024oracle}
{\sc K.~Deng, J.~Hu, J.~Wu, and Z.~Wen}, {\em Oracle complexities of augmented
  {L}agrangian methods for nonsmooth manifold optimization}, arXiv preprint
  arXiv:2404.05121,  (2024).

\bibitem{deng2022manifold}
{\sc K.~Deng and Z.~Peng}, {\em {A manifold inexact augmented {L}agrangian
  method for nonsmooth optimization on {R}iemannian submanifolds in Euclidean
  space}}, IMA Journal of Numerical Analysis,  (2022),
  \url{https://doi.org/10.1093/imanum/drac018}.

\bibitem{deng2025stochastic}
{\sc K.~Deng, S.~Zhang, B.~Wang, J.~Jin, J.~Zhou, and H.~Wang}, {\em Stochastic
  momentum {ADMM} for nonconvex and nonsmooth optimization with application to
  {PnP} algorithm}, arXiv preprint arXiv:2504.08223,  (2025).

\bibitem{gao2024non}
{\sc Y.~Gao, A.~Rodomanov, and S.~U. Stich}, {\em Non-convex stochastic
  composite optimization with {P}olyak momentum}, arXiv preprint
  arXiv:2403.02967,  (2024).

\bibitem{ghadimi2016accelerated}
{\sc S.~Ghadimi and G.~Lan}, {\em Accelerated gradient methods for nonconvex
  nonlinear and stochastic programming}, Mathematical Programming, 156 (2016),
  pp.~59--99.

\bibitem{han2020riemannian}
{\sc A.~Han and J.~Gao}, {\em Riemannian stochastic recursive momentum method
  for non-convex optimization}, arXiv preprint arXiv:2008.04555,  (2020).

\bibitem{han2021improved}
{\sc A.~Han and J.~Gao}, {\em Improved variance reduction methods for
  {R}iemannian non-convex optimization}, IEEE Transactions on Pattern Analysis
  and Machine Intelligence, 44 (2021), pp.~7610--7623.

\bibitem{ijcai2021p345}
{\sc A.~Han and J.~Gao}, {\em {Riemannian} stochastic recursive momentum method
  for non-convex optimization}, in Proceedings of the Thirtieth International
  Joint Conference on Artificial Intelligence, {IJCAI-21}, 2021,
  pp.~2505--2511.

\bibitem{hosseini2017riemannian}
{\sc S.~Hosseini and A.~Uschmajew}, {\em A {R}iemannian gradient sampling
  algorithm for nonsmooth optimization on manifolds}, SIAM Journal on
  Optimization, 27 (2017), pp.~173--189.

\bibitem{hu2020brief}
{\sc J.~Hu, X.~Liu, Z.-W. Wen, and Y.-X. Yuan}, {\em A brief introduction to
  manifold optimization}, Journal of the Operations Research Society of China,
  8 (2020), pp.~199--248.

\bibitem{huang2019faster}
{\sc F.~Huang, S.~Chen, and H.~Huang}, {\em Faster stochastic alternating
  direction method of multipliers for nonconvex optimization}, in International
  conference on machine learning, PMLR, 2019, pp.~2839--2848.

\bibitem{huang2016stochastic}
{\sc F.~Huang, S.~Chen, and Z.~Lu}, {\em Stochastic alternating direction
  method of multipliers with variance reduction for nonconvex optimization},
  arXiv preprint arXiv:1610.02758,  (2016).

\bibitem{jolliffe2003modified}
{\sc I.~T. Jolliffe, N.~T. Trendafilov, and M.~Uddin}, {\em A modified
  principal component technique based on the {LASSO}}, Journal of computational
  and Graphical Statistics, 12 (2003), pp.~531--547.

\bibitem{kasai2018riemannian}
{\sc H.~Kasai, H.~Sato, and B.~Mishra}, {\em Riemannian stochastic recursive
  gradient algorithm}, in International conference on machine learning, PMLR,
  2018, pp.~2516--2524.

\bibitem{levy2021storm+}
{\sc K.~Levy, A.~Kavis, and V.~Cevher}, {\em Storm+: Fully adaptive sgd with
  recursive momentum for nonconvex optimization}, Advances in Neural
  Information Processing Systems, 34 (2021), pp.~20571--20582.

\bibitem{li2021weakly}
{\sc X.~Li, S.~Chen, Z.~Deng, Q.~Qu, Z.~Zhu, and A.~Man-Cho~So}, {\em Weakly
  convex optimization over {S}tiefel manifold using {R}iemannian
  subgradient-type methods}, SIAM Journal on Optimization, 31 (2021),
  pp.~1605--1634.

\bibitem{li2021rate}
{\sc Z.~Li, P.-Y. Chen, S.~Liu, S.~Lu, and Y.~Xu}, {\em Rate-improved inexact
  augmented {L}agrangian method for constrained nonconvex optimization}, in
  International Conference on Artificial Intelligence and Statistics, PMLR,
  2021, pp.~2170--2178.

\bibitem{li2024stochastic}
{\sc Z.~Li, P.-Y. Chen, S.~Liu, S.~Lu, and Y.~Xu}, {\em Stochastic inexact
  augmented {L}agrangian method for nonconvex expectation constrained
  optimization}, Computational Optimization and Applications, 87 (2024),
  pp.~117--147.

\bibitem{lu2024variance}
{\sc Z.~Lu, S.~Mei, and Y.~Xiao}, {\em Variance-reduced first-order methods for
  deterministically constrained stochastic nonconvex optimization with strong
  convergence guarantees}, arXiv preprint arXiv:2409.09906,  (2024).

\bibitem{peng2022riemannian}
{\sc Z.~Peng, W.~Wu, J.~Hu, and K.~Deng}, {\em Riemannian smoothing gradient
  type algorithms for nonsmooth optimization problem on compact {R}iemannian
  submanifold embedded in {E}uclidean space}, Applied Mathematics \&
  Optimization, 88 (2023), p.~85.

\bibitem{sahin2019inexact}
{\sc M.~F. Sahin, A.~Eftekhari, A.~Alacaoglu, F.~L. G{\'o}mez, and V.~Cevher},
  {\em An inexact augmented {L}agrangian framework for nonconvex optimization
  with nonlinear constraints}, in Proceedings of NeurIPS 2019, 2019.

\bibitem{sato2019riemannian}
{\sc H.~Sato, H.~Kasai, and B.~Mishra}, {\em Riemannian stochastic variance
  reduced gradient algorithm with retraction and vector transport}, SIAM
  Journal on Optimization, 29 (2019), pp.~1444--1472.

\bibitem{selvan2013spherical}
{\sc S.~E. Selvan, P.~B. Borckmans, A.~Chattopadhyay, and P.-A. Absil}, {\em
  Spherical mesh adaptive direct search for separating quasi-uncorrelated
  sources by range-based independent component analysis}, Neural computation,
  25 (2013), pp.~2486--2522.

\bibitem{selvan2015range}
{\sc S.~E. Selvan, S.~T. George, and R.~Balakrishnan}, {\em Range-based {ICA}
  using a nonsmooth quasi-{N}ewton optimizer for electroencephalographic source
  localization in focal epilepsy}, Neural computation, 27 (2015), pp.~628--671.

\bibitem{shi2025momentum}
{\sc Q.~Shi, X.~Wang, and H.~Wang}, {\em A momentum-based linearized augmented
  {L}agrangian method for nonconvex constrained stochastic optimization},
  Mathematics of Operations Research,  (2025).

\bibitem{tran2022hybrid}
{\sc Q.~Tran-Dinh, N.~H. Pham, D.~T. Phan, and L.~M. Nguyen}, {\em A hybrid
  stochastic optimization framework for composite nonconvex optimization},
  Mathematical Programming, 191 (2022), pp.~1005--1071.

\bibitem{wang2022riemannian}
{\sc B.~Wang, S.~Ma, and L.~Xue}, {\em Riemannian stochastic proximal gradient
  methods for nonsmooth optimization over the {S}tiefel manifold}, Journal of
  Machine Learning Research, 23 (2022), pp.~1--33.

\bibitem{wang2017penalty}
{\sc X.~Wang, S.~Ma, and Y.-x. Yuan}, {\em Penalty methods with stochastic
  approximation for stochastic nonlinear programming}, Mathematics of
  computation, 86 (2017), pp.~1793--1820.

\bibitem{wang2019spiderboost}
{\sc Z.~Wang, K.~Ji, Y.~Zhou, Y.~Liang, and V.~Tarokh}, {\em Spiderboost and
  momentum: Faster variance reduction algorithms}, Advances in Neural
  Information Processing Systems, 32 (2019).

\bibitem{xu2023momentum}
{\sc Y.~Xu and Y.~Xu}, {\em Momentum-based variance-reduced proximal stochastic
  gradient method for composite nonconvex stochastic optimization}, Journal of
  Optimization Theory and Applications, 196 (2023), pp.~266--297.

\bibitem{zhou2019faster}
{\sc P.~Zhou, X.-T. Yuan, and J.~Feng}, {\em Faster first-order methods for
  stochastic non-convex optimization on {R}iemannian manifolds}, in The 22nd
  international conference on artificial intelligence and statistics, PMLR,
  2019, pp.~138--147.

\end{thebibliography}
\end{document}